\documentclass[11pt]{amsart}
\usepackage{amsmath}
\usepackage{amsfonts}
\usepackage{amssymb}
\usepackage{amsthm}
\usepackage{mathtools} 
\usepackage{latexsym}

\usepackage{enumerate}
\usepackage{cancel}
\usepackage{cases}
\usepackage{empheq}
\usepackage{multicol}

  \usepackage{caption}
  \usepackage{subcaption}
  \usepackage{graphics} 
  \usepackage{epsfig} 
\usepackage{graphicx} 

\usepackage{float}

\usepackage{color}

\usepackage[backgroundcolor=gray!30,linecolor=black]{todonotes}
   
\usepackage{mathrsfs}
\usepackage{fontenc} 
\usepackage{inputenc} 

\usepackage{verbatim}

\theoremstyle{plain}
\newtheorem{theorem}{Theorem}[section]
\newtheorem{proposition}[theorem]{Proposition}
\newtheorem{lemma}[theorem]{Lemma}

\theoremstyle{definition}
\newtheorem{definition}[theorem]{Definition}

\theoremstyle{remark}
\newtheorem{remark}[theorem]{Remark}

\numberwithin{equation}{section} 
\numberwithin{figure}{section}   

\usepackage[square,comma,numbers,sort]{natbib}
\usepackage[colorlinks=true, pdfborder={ 0 0 0}]{hyperref}
\hypersetup{urlcolor=blue, citecolor=red}
\usepackage{url}

\newcommand{\vect}[1]{\mathbf{#1}}

\newcommand{\bu}{\vect{u}}
\newcommand{\bv}{\vect{v}}
\newcommand{\bw}{\vect{w}}

\newcommand{\bx}{\vect{x}}

\newcommand{\bbf}{\vect{f}}

\newcommand{\field}[1]{\mathbb{#1}}
\newcommand{\nN}{\field{N}}

\newcommand{\nR}{\field{R}}

\newcommand{\tu}{\Tilde{\bu}}

\newcommand{\bomega}{\boldsymbol{\omega}}

\newcommand{\nT}{\mathbb T}

\newcommand{\lp}{\left(}
\newcommand{\rp}{\right)}
\newcommand{\dt}{ \partial_t }
\newcommand{\lap}{\triangle}
\newcommand{\grad}{\nabla}
\newcommand{\bze}{\boldsymbol{\zeta}^\epsilon}
\newcommand{\bzd}{\boldsymbol{\zeta}^\delta}
\newcommand{\R}{\mathbb{R}}

\newcommand{\pd}[2]{\frac{\partial #1}{\partial #2}}
\newcommand{\abs}[1]{\left\lvert#1\right\rvert}
\newcommand{\set}[1]{\left\{#1\right\}}

\newcommand{\norm}[1]{\left\|#1\right\|}

\newcommand{\normLp}[2]{\left\|#2 \right\|_{L^{#1}}}
\newcommand{\normHs}[2]{\left\|#2 \right\|_{H^{#1}}}

\usepackage{url}

\usepackage{placeins} 

\usepackage{enumitem} 
\newcounter{my_counter}
\title[Calmed Navier-Stokes]
{Calmed 3D Navier-Stokes Equations: Global Well-Posedness, Energy Identities, Global Attractors, and Convergence}

\date{\today}

\author{Matthew Enlow}
\address[Matthew Enlow]{Department of Mathematics, 
                University of Nebraska--Lincoln,
        Lincoln, NE 68588-0130, USA}
\email[Matthew Enlow]{menlow2@huskers.unl.edu}

\author{Adam Larios}
\address[Adam Larios]{Department of Mathematics, 
                University of Nebraska--Lincoln,
        Lincoln, NE 68588-0130, USA}
\email[Adam Larios]{alarios@unl.edu}

\author{Jiahong Wu}
\address[Jiahong Wu]{Department of Mathematics, 
                University of Notre Dame
        Notre Dame, IN 46556, USA}
\email[Jiahong Wu]{jwu29@nd.edu}


\begin{document}
\begin{abstract}
We propose a modification to the nonlinear term of the three-dimensional incompressible Navier-Stokes equations (NSE) in either advective or rotational form which ``calms'' the system in the sense that the algebraic degree of the nonlinearity is effectively reduced. 
This system, the calmed Navier-Stokes Equations (calmed NSE), utilizes a ``calming function'' in the nonlinear term to locally constrain large advective velocities. Notably, this approach avoids the direct smoothing or filtering of derivatives, thus we make no modifications to the boundary conditions.
Under suitable conditions on the calming function, we are able to prove global well-posedness of calmed NSE and show the convergence of calmed NSE solutions to NSE solutions on the time interval of existence for the latter.  In addition, we prove that the dynamical system generated by the calmed NSE in the rotational form possesses both an energy identity and a global attractor. Moreover, we show that strong solutions to the calmed equations converge to strong solutions of the NSE without assuming their existence, providing a new proof of the existence of strong solutions to the 3D Navier-Stokes equations.
\end{abstract}

\maketitle
\thispagestyle{empty}

\section{Introduction}\label{secInt}

\noindent
Is smoothing the only method to control the Navier-Stokes equations? A major obstacle in proving global well-posedness for the 3D Navier-Stokes equations in fluid dynamics is the rapid intensification of small length scales. Consequently, many approaches have focused on mitigating this growth by introducing stronger diffusion, or by mollifying or filtering the nonlinear term. These strategies essentially involve some form of smoothing. However, derivatives can also grow via another mechanism: multiplication, which can lead to the generation of smaller length scales\footnote{For example, consider $g(x) = \sin(x) + \cos(x)$. It is straightforward to show that $\|\frac{d^n}{dx^n}g\|_{L^\infty} = \sqrt{2}$ for all $n\in\nN$, but $\|\frac{d}{dx}g^{n}\|_{L^\infty} \geq n$, and hence $\frac{d}{dx}g^{n}$ grows without bound as $n\to\infty$.}. In this work, we introduce a novel modification to the incompressible Navier-Stokes equations (NSE) that tempers the effect of the algebraic multiplication without introducing a smoothing operator. Specifically, we limit the advective velocity by smoothly truncating it, a process we call ``algebraic calming'' or simply ``calming,'' since it effectively reduces the algebraic degree of the nonlinearity.  Calming was introduced by the authors of the present work in \cite{Enlow_Larios_Wu_2023_calming} in the context of the 2D Kuramoto-Sivashinsky equations (KSE).  In the present work, we propose and study two calmed versions of the 3D Navier-Stokes equations, which we call the ``calmed Navier-Stokes equations,'' (calmed NSE).

Calming has several advantages over smoothing; namely:
\begin{itemize}
    \item There is no need to modify the boundary conditions, the system is globally well-posed, in both 2D and 3D, with standard homogeneous Dirichlet (i.e., ``no-slip'') boundary conditions.
    \item The system is of the same derivative order as the Navier-Stokes equations, as there are no modifications to the derivatives introduced.
    \item The ``calming'' modification is an entirely local operation, which may be more efficient than, e.g., mollification or filtering in computational settings, especially in the setting of parallel processing.  (There is also no auxilliary equation to handle, such as in the case of the $k-\epsilon$ or $k-\omega$ models.)
    \item Solutions of a rotational version of the calmed model satisfy exactly the same energy identity as that of strong solutions to the Navier-Stokes equations.
\end{itemize}

Specifically, in the present work, we prove that the calmed NSE are globally well-posed in 3D with no-slip (i.e., physical) boundary conditions, and that their solutions converge, as the calming parameter $\epsilon\rightarrow0^+$, to strong solutions of the Navier-Stokes equations on the time interval of existence and uniqueness of the latter.  Moreover, we propose a version of this calming modification in the context of the so-called rotational form of the NSE, where the calming is applied only to the rotational form of the nonlinearity. For this system, which we call the ``calmed rotational Navier-Stokes equations'' (calmed rNSE), we prove that under an additional assumption on the calming function, the resulting system satisfies exactly the same energy equality as for strong solutions to the NSE, in addition to enjoying the aforementioned properties of the calmed NSE.  We then use this energy equality to prove that the calmed rNSE has a compact global attractor.
We also prove that solutions of the calmed system converge to strong solutions of the original Navier-Stokes system (at least, before the potential blow-up time of the later).  In addition to this, we show that there is no need to assume the existence of strong solutions to the Navier-Stokes equations \textit{a priori}.  In particular, via calming, we provide a new independent proof of the existence of strong solutions to the 3D Navier-Stokes equations.

The three-dimensional (3D) incompressible constant-density Navier-Stokes equations (NSE) are given by
\begin{subequations}\label{NSE}
\begin{empheq}[left=\empheqlbrace]{align}
\label{NSE_mo}
\dt\bu + (\bu\cdot\nabla)\bu + \grad p  &= \nu \triangle\bu + \mathbf{f}&&\text{ in }\Omega\times(0,T),\\
\nabla\cdot\bu &= 0&&\text{ in }\Omega\times(0,T),\label{NSE_div}
\\
\bu\big|_{\partial\Omega} &= \mathbf{0} &&\text{ on }\partial\Omega\times(0,T)
\\
\bu(\bx,0)&=\bu_0(\bx)&&\text{ in }\Omega, 
\end{empheq}
\end{subequations}
Here, $\bu:\Omega\times[0,T] \to \R^3$ is the fluid velocity,  $p:\Omega\times[0,T] \to \R$ is the (kinematic) pressure, and $\bbf:\Omega\times[0,T] \to \R^3$ is a body force. The domain $\Omega\subset\nR^3$ is a bounded, open, connected set with $C^2$ boundary.  

Note that, using the vector identity 
\begin{align}\label{vec_identity}
 (\bu\cdot\nabla)\bu = (\nabla\times\bu)\times\bu + \tfrac12\nabla|\bu|^2,
\end{align}
one may formally rewrite \eqref{NSE} in the following equivalent rotational form (rNSE),
\begin{subequations}\label{NSE_vor}
\begin{empheq}[left=\empheqlbrace]{align}
\dt\bu + \bomega\times\bu  +\nabla \pi &= \nu \triangle\bu + \mathbf{f} &&\text{ in }\Omega\times(0,T),\\
\nabla\cdot\bu &= 0,\label{NSE_vor_div} &&\text{ in }\Omega\times(0,T),\\ 
\bu\big|_{\partial\Omega} &= \mathbf{0} &&\text{ on }\partial\Omega\times(0,T), \\
\bu(\bx,0)&=\bu_0(\bx)&&\text{ in }\Omega, 
\end{empheq}
\end{subequations}
where we have denoted the vorticity by $\bomega :=\nabla\times\bu$ and the Bernoulli pressure (or ``dynamic pressure'') as $\pi:=p + \tfrac12|\bu|^2$. The term $\bomega\times \bu$ is sometimes called the Lamb vector.

Using the techniques introduced in the recent paper \cite{Enlow_Larios_Wu_2023_calming}, we use a bounded smooth truncation function --- that we call a ``calming function'' when used in this context --- that approximates the identity as the ``calming parameter'' $\epsilon\rightarrow0^+$. We propose a calming-function approach to the 3D NSE.  To make things concrete, we consider several forms of calming functions (the first three of which were also considered in \cite{Enlow_Larios_Wu_2023_calming}); namely,
\begin{align} \label{zeta_choices}
\bze(\bx) = 
\begin{cases}
  \bze_1(\bx) := \frac{\bx}{1+\epsilon|\bx|}, &\text{ or}\\
  \bze_2(\bx) := \frac{\bx}{1+\epsilon^2|\bx|^2}, &\text{ or}\\
  \bze_3(\bx) := \frac{1}{\epsilon}\arctan(\epsilon\bx), &\text{ or}\\
  \bze_4(\bx) := q^\epsilon(\abs{\bx}) \frac{\bx}{\abs{\bx}} , &
\end{cases}
\end{align}
where the arctangent in $\bze_3$ acts component-wise; 
\[\arctan\lp (z_1,z_2,z_3)^T \rp= ( \arctan(z_1), \arctan(z_2),\arctan(z_3))^T,\] 
and for $\bze_4$, we define $\bze_4(\boldsymbol{0})=\boldsymbol{0}$, and

\begin{align}\label{calm4}
q^\epsilon(r) = 
    \begin{cases}
        r, \quad  
        & 0 \leq  r < \frac{1}{\epsilon}, \\ 
        -\frac{\epsilon}{2}\lp r - \frac{2}{\epsilon} \rp^2 + 
        \frac{3}{2\epsilon}, \quad  
        & \frac{1}{\epsilon} \leq r < \frac{2}{\epsilon}, \\ 
        \frac{3}{2\epsilon} , \quad 
        & r \geq  \frac{2}{\epsilon}. 
    \end{cases}
\end{align}

Note that $\bze(\bx)\rightarrow\bx$ for all $\bx\in\Omega$ (i.e.,  pointwise) as $\epsilon\rightarrow0^+$, and $\bze_i\in C^1$ for $i=1,\ldots,4$. 
We also require that, for example, $\bze$ be bounded for $\epsilon > 0$ fixed. We describe in detail the conditions we assume for $\bze$ in Definition \ref{zeta_def}.

The idea is that, when used in the nonlinear term, a calming function allows for control over the $L^\infty$ norm that is otherwise unavailable. This permits a proof of global well-posedness of the modified system without the need to modify boundary conditions or add higher-order derivatives (as in, e.g., modifications based on adding higher-order viscosity \cite{Layton_Rebholz_2013_Voigt}). Moreover, we can prove that, at least before the potential blow-up time of the original PDE, solutions to the modified PDE converge to solutions to the original PDE as $\epsilon\rightarrow 0$.  At least, this is the program that was carried out in \cite{Enlow_Larios_Wu_2023_calming} in the context of the 2D Kuramoto-Sivashsinsky equation.  In the present work, we extend this program to the 3D Navier-Stokes equations.

In particular, we propose two modifications of the Navier-Stokes system.  The first is based on the form \eqref{NSE}. Continuing the same approach we employed in \cite{Enlow_Larios_Wu_2023_calming}, we introduce the following system that we call the calmed Navier-Stokes equations (calmed NSE).
\begin{subequations}\label{cNSE}
\begin{empheq}[left=\empheqlbrace]{alignat=3}
\label{cNSE_mo}
\dt\bu + (\bze(\bu)\cdot\nabla)\bu + \grad p  &= \nu \triangle\bu + \mathbf{f}&&\text{ in }\Omega\times(0,T),\\
\nabla\cdot\bu &= 0&&\text{ in }\Omega\times(0,T),\label{cNSE_div}
\\
\bu\big|_{\partial\Omega} &= \mathbf{0} &&\text{ on }\partial\Omega\times(0,T)
\\
\bu(\bx,0)&=\bu_0(\bx)&&\text{ in }\Omega, 
\end{empheq}
\end{subequations}

One can see \eqref{cNSE} as a modification of \eqref{NSE} in the spirit of Leray (see, e.g., \cite{Leray_1934essai, Yamazaki_2012_Leray_alpha, Farhat_Lunasin_Titi_2018_Leray_AOT, Cheskidov_Holm_Olson_Titi_2005, Cao_Titi_2009, Ilyin_Lunasin_Titi_2006, Hecht_Holm_Petersen_Wingate_2008_analysis, Cao_Holm_Titi_2005_Jturb, Chen_Foias_Holm_Olson_Titi_Wynne_1999} and many others), except that our modification does not mollify the nonlinearity but is instead a local truncation of the advective velocity. 

While we show in the present work that the calming modification of \eqref{cNSE} allows for a proof of global well-posedness and other desirable properties, it is clear that such a modification would have a different energy balance than that of Navier-Stokes, as the nonlinear term does not vanish in standard energy calculations.
Therefore, we also consider a related modification of the rotational form \eqref{NSE_vor} which locally limits the strength of the rotational term.  Namely, we propose the following system, which we call the calmed rotational Navier-Stokes equations (calmed rNSE).
\begin{subequations}\label{cNSE_vor}
\begin{empheq}[left=\empheqlbrace]{alignat=3}
\label{cNSE_vor_mo}
\dt\bu  + (\nabla\times\bu)\times\bze(\bu) + \grad \pi  &= \nu \triangle\bu + \mathbf{f}&&\text{ in }\Omega\times(0,T),\\
\nabla\cdot\bu &= 0&&\text{ in }\Omega\times(0,T),\label{cNSE_vor_div}
\\
\bu\big|_{\partial\Omega} &= \mathbf{0} &&\text{ on }\partial\Omega\times(0,T)
\\
\bu(\bx,0)&=\bu_0(\bx)&&\text{ in }\Omega. 
\end{empheq}
\end{subequations}
Due to the presence of the calming function, one cannot rewrite calmed NSE as calmed rNSE using \eqref{vec_identity} as we do for NSE and rNSE. Thus, while they are both modifications of the Navier-Stokes equations which are similar, we treat them as different systems. 
However, system \eqref{cNSE_vor} is an interesting object to study in its own right. Thanks to the well-known geometric identity for the cross product,
\begin{align}\label{cross_prod_ortho}
(\mathbf{A}\times\mathbf{B})\cdot\mathbf{B}=0,
\end{align}
one discovers exceptional features of System \eqref{cNSE_vor} when $\bze$ is suitably chosen.
Namely, when $\bze(\bx)$ can be expressed as a scalar multiple of $\bx$ pointwise we deduce that \eqref{cNSE_vor} possesses both an energy identity (Theorem \ref{energy_eq}) and its dynamical system has a global attractor (Theorem \ref{thm_cNSE_vort_Glatt}).

\begin{remark}
Applying a bounded truncation operator to the nonlinear term in $3$D Navier Stokes was also considered by Yoshida and Giga \cite{Yoshida_1984nonlinear} and by the authors of \cite{Caraballo_Real_Kloeden_2006_ANS} in the study of the globally modified Navier-Stokes Equations (GMNSE) (see also,
\cite{Caraballo_Kloeden_Real_2008_DCDS,
Chai_Duan_2019_attractor_GMNSE_fractional,
Deugou_Medjo_2018_CPAA,
Kloeden_Langa_Real_2007_CPAA,
Kloeden_Rubio_Real_2009_CPAA,
Romito_2009_ANS,
Rubio_Duran_Real_2011_DCDS,
Zhang_2009_tamed_NSE,
Zhao_Yang_2017_GMNSE_pullback}).
In those works, the following system was studied.
\begin{subequations}
\begin{empheq}[left=\empheqlbrace]{alignat=3}
\notag
\dt{\bu}+\min\set{1,N\|\nabla\bu\|_{L^2}^{-1}}(\bu\cdot\nabla)\bu +\nabla p &= \nu\triangle\bu + \bbf &&\text{ in }\Omega\times(0,T),\\ \notag
\nabla\cdot\bu &= 0&&\text{ in }\Omega\times(0,T),
\\ \notag
\bu\big|_{\partial\Omega} &= \mathbf{0} &&\text{ on }\partial\Omega\times(0,T)
\\ \notag
\bu(\bx,0)&=\bu_0(\bx)&&\text{ in }\Omega, 
\end{empheq}
\end{subequations}
For GMNSE, solutions converge to a solution of $3$D Navier-Stokes as parameter $N$ tends to infinity. 
This system is similar to calmed NSE \eqref{cNSE} in that it bounds the nonlinear term as the velocity $\bu$ gets large in a certain sense. However, our modification has several advantages over GMNSE. Namely, that the calming functions in the present work are defined pointwise and only bound the solution $\bu$ in regions where $|\bu(\bx,t)|$ is greater than approximately $\epsilon^{-1}$, whereas the modification in GMNSE affects the solution globally. Also, whenever $\normLp{2}{\grad\bu} \to \infty$, the nonlinearity in GMNSE vanishes entirely, but for calmed NSE this would only cause the large values of $\abs{\bu(\bx,t)}$ to be truncated locally.
Moreover, our calming parameter depends on $\bu$ while the GM function depends on $\grad\bu$, hence the manner in which we control the nonlinearity is different.
In a future work, we will examine differences between these two systems computationally.
\end{remark}

\subsection{Main results}\label{subsec_Main_Results}
The proofs of existence, uniqueness, convergence, etc. are essentially identical for both equations \eqref{cNSE} and \eqref{cNSE_vor}.  (The only phenomenological difference examined in this paper is the rotational form \eqref{cNSE_vor} has an energy identity, but for \eqref{cNSE}, this is unknown.)  Therefore, we adopt a unified abstract notation which allows us to handle both equations simultaneously.  
For either \eqref{cNSE} or \eqref{cNSE_vor}, the weak formulation can be written as follows: Given $\bu_0 \in L^2(0,T; H)$ and $\bbf \in L^2(0,T; V')$, find $\bu \in L^2(0,T; V)$ which satisfies 

\begin{subequations}\label{cNSE_weak}
\begin{align}
\label{cNSE_weak_eqn} 
\left\langle \dt\bu, \bv \right\rangle + 
\left\langle \nu A\bu, \bv \right\rangle + 
\left\langle B(\bze(\bu), \bu), \bv \right\rangle &= 
\left\langle \bbf, \bv \right \rangle \quad \text{ for all } \bv\in V, \\
\label{cNSE_weak_IC}    \bu(\bx, 0) &= \bu_0(\bx),
\end{align}
\end{subequations}
where the Stokes operator $A$ is defined by \eqref{Stokes} and the nonlinear term $B(\cdot, \cdot)$ is defined in \eqref{nonlinear_bi} in either advective or rotational form.

We note that the uniqueness of weak solutions is an open problem, similar to the situation regarding $3$D Navier-Stokes. However, unlike the $3$D Navier-Stokes case, we are able to prove the global existence of strong solutions.

First, we give precise conditions on the types of calming functions we allow.

\begin{definition}\label{zeta_def}
    We say $\bze: \R^3 \to \R^3$ is a \emph{calming function} if the following three conditions hold:
    \begin{enumerate}
    \item \label{zeta_cond_Lip} $\bze$ is Lipschitz continuous with Lipschitz constant $1$.
    \item \label{zeta_cond_bdd} For $\epsilon > 0$ fixed, $\bze$ is bounded.
    \item \label{zeta_cond_conv} There exists $C > 0$, $\alpha >0$ and $\beta \geq 1$ such that for any $\bx \in \R^3$,
    \begin{align}\label{zeta_pwise_conv}
        \abs{\bze(\bx) - \bx} \leq C \epsilon^\alpha \abs{\bx}^\beta
    \end{align}
    \end{enumerate}
(with additional constraints stated in the theorem on $\alpha$ and $\beta $). 
For the energy equality, we also require:
   \begin{enumerate}[resume] 
    \item \label{zeta_cond_parll} 
     For any $\epsilon >0$ and for each $\bx\in \R^3$ there exists $\lambda^\epsilon(\bx) \in \R$ such that $\bze(\bx) = \lambda^\epsilon(\bx) \bx$.
     That is, $\bze(\bx)$ is parallel to $\bx$.
    \end{enumerate}
\end{definition}

\begin{remark}
    The lower bound on $\beta$ is necessary to satisfy condition \ref{zeta_cond_bdd} and inequality \eqref{zeta_pwise_conv} of $\bze$. Using the triangle inequality and \eqref{zeta_pwise_conv}, we may write 
    \begin{align}\notag
        \abs{\bx} \leq 
        C\epsilon^\alpha \abs{\bx}^\beta + 
        \normLp{\infty}{\bze},
    \end{align}
    which, when $\abs{\bx}$ is sufficiently large, fails to be valid for $\beta < 1$. 
\end{remark}
    
We indicate in the next proposition the extent to which our examples of a calming function (stated in \eqref{zeta_choices}) satisfy the conditions of Definition \ref{zeta_def}.

\begin{proposition}\label{zeta_cond_prop}
    Consider $\bze_i$ as described in \eqref{zeta_choices}. 
    
    For $i = 1,2,4$,
    $\bze_i$ satisfies Conditions \ref{zeta_cond_Lip}-\ref{zeta_cond_parll} of Definition \ref{zeta_def}. For $i=3$, $\bze_i$ satisfies Conditions \ref{zeta_cond_Lip}, \ref{zeta_cond_bdd}, and \ref{zeta_cond_conv} of Definition \ref{zeta_def}. In particular, the following explicit bounds hold for $\epsilon>0$.
    \begin{enumerate}
        \item For $\bze_1$, 
        \[
        \normLp{\infty}{ \bze_1} = \frac{1}{\epsilon}
        \text{ and } \abs{\bze_1(\bx) - \bx} \leq \epsilon \abs{\bx}^2.
        \]
        \item For $\bze_2$, 
        \[
        \normLp{\infty}{ \bze_2} = \frac{1}{2\epsilon}
        \text{ and } \abs{\bze_2(\bx) - \bx} \leq \epsilon^2 \abs{\bx}^3.
        \]
        \item For $\bze_3$, 
        \[
        \normLp{\infty}{ \bze_3} = \frac{\sqrt{3}\pi}{2\epsilon}
        \text{ and } \abs{\bze_3(\bx) - \bx} \leq \epsilon^2 \abs{\bx}^3.
        \]
        \item For $\bze_4$, 
        \[
        \normLp{\infty}{ \bze_4} = \frac{3}{2\epsilon} 
        \text{ and } \abs{\bze_4(\bx) - \bx} \leq \epsilon \abs{\bx}^2.
        \]
    \end{enumerate}
\end{proposition}

Next, we begin by defining what we mean by weak and strong solutions.

\begin{definition}[Weak solution]\label{cNSE_vel_weak_soln}
Let $T>0$, $\bu_0 \in H$ and let $\bbf \in L^2(0,T;V')$. We say that $\bu$ is a \textit{weak solution} to calmed NSE \eqref{cNSE} or calmed rNSE \eqref{cNSE_vor} on the interval $[0,T]$ if $\bu$ satisfies equation \eqref{cNSE_weak_eqn} for all $\bv\in V$ in the sense of $L^2((0,T))$ with $\bu\in C([0,T];H)$ and $\dt\bu \in L^2(0,T;V')$. Furthermore, we require \eqref{cNSE_weak_IC} to be satisfied in the sense of $C([0,T];H)$.
\end{definition}

\begin{definition}[Strong solution] \label{cNSE_vel_strong_soln}
    Let $T>0$, $\bu_0\in V$, and let $\bbf \in L^2(0,T;H)$. We say that $\bu$ is a \textit{strong solution} to calmed NSE \eqref{cNSE} or calmed rNSE \eqref{cNSE_vor} on the interval $[0,T]$ if $\bu$ is a weak solution and also 
    $\bu \in C([0,T]; V) \cap L^2(0,T; H^2\cap V)$ with time derivative $\dt \bu \in L^2(0,T; H)$ and initial data satisfied in the sense of $C([0,T]; V)$.
\end{definition}

We now state our results on the global well-posedness of solutions to calmed Navier-Stokes and calmed rotational Navier-Stokes. 

\begin{theorem}[Global existence of weak solutions to calmed systems]\label{thm_cNSE_vel_exi}
 Let $\bu_0 \in H$, $T>0$, and let $\bbf \in L^2(0,T;V')$ be given. Suppose, for $\epsilon > 0$, $\bze$ is a calming function which satisfies conditions \ref{zeta_cond_Lip}, \ref{zeta_cond_bdd}, and \ref{zeta_cond_conv} of Definition \ref{zeta_def}. Then weak solutions to calmed NSE or calmed rNSE \eqref{cNSE_weak} exist on $[0,T]$.
\end{theorem}

\begin{theorem}[First-order regularity of calmed systems]\label{thm_cNSE_vel_reg}
 Let $T>0$. Suppose that $\bu_0 \in V$ and that $\bbf \in L^2(0,T;H)$. Consider a weak solution $\bu$ to calmed NSE or calmed rNSE $\eqref{cNSE_weak}$ on the interval $[0,T]$. Then $\bu \in C([0,T]; V) \cap L^2(0,T; H^2\cap V)$ and $\dt \bu \in L^2(0,T; H)$.
\end{theorem}

\begin{theorem}[Global well-posedness of strong solutions to calmed systems]\label{thm_cNSE_vel_uni}
    Let $T>0$, $\bu_0, \in V$, and let $\bbf \in L^2(0,T; H)$. Suppose $\bze$ is a calming function which satisfies conditions \ref{zeta_cond_Lip}, \ref{zeta_cond_bdd}, and \ref{zeta_cond_conv} of Definition \ref{zeta_def}. Then there exists a strong solution $\bu \in C([0,T];V)\cap L^2(0,T; H^2\cap V)$ to calmed NSE \eqref{cNSE} and calmed rNSE \eqref{cNSE_vor} which depends continuously on its initial data and is unique in the class of weak solutions. 
\end{theorem}

For our calmed systems we also have the convergence of \eqref{cNSE} (resp. \eqref{cNSE_vor}) to \eqref{NSE} (resp. \eqref{NSE_vor}) on short time intervals.

\begin{theorem}[Convergence]\label{thm_cNSE_conv}
    Let $T>0$, and let $\bze$ be a calming function satisfying conditions \ref{zeta_cond_Lip}, \ref{zeta_cond_bdd}, and \ref{zeta_cond_conv} of Definition \ref{zeta_def}, where $\beta \geq 1$ is the minimal value for which \ref{zeta_cond_conv} holds. 
    Suppose
    \begin{align}\label{NSE_reg_cond}
        \bu \in C([0,T]; V)\cap L^{2\beta}(0,T; H^2\cap V)
    \end{align} 
    is a strong solution to the $3$D Navier-Stokes equation written either in its velocity form \eqref{NSE} or rotational form \eqref{NSE_vor} with initial data $\bu_0 \in V$ and forcing term $\bbf \in L^2(0,T;H)$.
    Suppose $\bu^\epsilon \in C([0,T]; V)\cap L^2(0,T; H^2\cap V)$ is a solution to the $3$D calmed NSE \eqref{cNSE} (resp. $3$D calmed rNSE \eqref{cNSE_vor}) with the same initial data $\bu_0$ and forcing term $\bbf$. Then 
    \begin{align}
        \| \bu - \bu^\epsilon \|_{L^\infty V} \leq K \epsilon^\alpha,
    \end{align}
    where $K>0$ is a constant depending only on $\Omega, \nu, \beta, \|\bu\|_{L^\infty V}, \normLp{2}{\lap\bu}, T,$ and $\alpha, \beta$ are determined by the choice of $\bze$ and are as given by condition \ref{zeta_cond_conv} of Definition \ref{zeta_def}.
\end{theorem}

In section \ref{sec_crNSE_cauchy} we show that in fact, strong solutions to the calmed systems are Cauchy with respect to the calming parameter $\epsilon > 0$ and that the limit point obtained from this sequence is itself a strong solution to $3$D Navier-Stokes.

\begin{theorem}
[Existence of Strong Solutions to Navier-Stokes]
\label{thm_cNSE_Cauchy}
    For each $\epsilon > 0$, let $\bze$ be a calming function satisfying conditions \ref{zeta_cond_Lip}, \ref{zeta_cond_bdd}, and \ref{zeta_cond_conv} of Definition \ref{zeta_def} and let $\bu^\epsilon$ be a strong solution to calmed NSE \eqref{cNSE} or calmed rNSE \eqref{cNSE_vor} with initial data $\bu_0 \in V$ and forcing term $\bbf\in L^\infty(0,\infty; L^2)$. Suppose $T > 0$ is the maximal time for which 
    \begin{align}
        \notag
        \sup_{\epsilon > 0}\sup_{0\leq t \leq T}\normLp{2}{\grad\bu^\epsilon(t)} \leq \sqrt{2}\normLp{2}{\grad\bu_0}
    \end{align}
    is valid. Then, 
    \begin{enumerate}
        \item The sequence $\{\bu^\epsilon \}_{\epsilon > 0}$ is Cauchy in $L^\infty H \cap L^2 V$.
        \item The limit point of the sequence, $\bu$, is a strong solution to the $3$D Navier-Stokes equations \eqref{NSE} or \eqref{NSE_vor} on the interval $[0,T]$.
    \end{enumerate}
\end{theorem}

While calmed NSE and calmed rNSE share the same properties of global well-posedness, we can see a key distinction between the two in the next theorem when $\bze$ is assumed to be pointwise parallel.

\begin{theorem}[Energy identity for weak solutions of calmed rNSE \eqref{cNSE_vor}] \label{thm_energy_identity}
Let $\nu>0$, $\epsilon>0$, $\bu_0\in H$, and $\bbf \in L^2(0,T;V')$ be given.   Suppose $\bze$ satisfies conditions \ref{zeta_cond_Lip}, \ref{zeta_cond_bdd}, \ref{zeta_cond_conv}, and \ref{zeta_cond_parll} of Definition \ref{zeta_def}.
Let $\bu^\epsilon$ be a weak solution to calmed rNSE \eqref{cNSE_vor}.  
Then $\bu^\epsilon$ the satisfies the energy equalities 
\begin{align}\label{energy_eq}
    \frac{1}{2}\frac{d}{dt}\normLp{2}{\bu^\epsilon}^2 + \nu \normLp{2}{\grad\bu^\epsilon}^2 = \left\langle \bbf, \bu^\epsilon \right\rangle.
\end{align}
and 
\begin{align}\label{energy_eq_integrated}
\normLp{2}{\bu^\epsilon(t)}^2 + 2\nu\int_0^t \normLp{2}{\grad\bu^\epsilon(s)}^2\,ds = \normLp{2}{\bu_0}^2 +2\int_0^t\left\langle \bbf(s), \bu^\epsilon(s) \right\rangle\,ds.
\end{align}
\end{theorem}
\begin{remark}
    Combining Theorems \ref{thm_cNSE_conv} and \ref{thm_energy_identity}, one can easily show that strong solutions to the Navier-Stokes equations enjoy an energy equality, by passing to a limit as $\epsilon\rightarrow0$ in \eqref{energy_eq_integrated}.  Hence, our approach can be seen as an alternate proof of this well-known fact.
\end{remark}

From these energy identities we deduce the existence of a global attractor, under the assumption that $\bbf$ is time-independent.

\begin{theorem}[Existence of a global attractor]\label{thm_cNSE_vort_Glatt}
 Let $\bze$ be a calming function which satisfies conditions 
 \ref{zeta_cond_Lip},
 \ref{zeta_cond_bdd}, \ref{zeta_cond_conv}, and \ref{zeta_cond_parll} of Definition \ref{zeta_def}. If $\bu_0 \in H$ and $\bbf \in H$ then the dynamical system on $H$ generated by calmed rNSE \eqref{cNSE_vor} has a global attractor $\mathcal{A}$ on $H$.
\end{theorem}

\begin{remark}
All of the above results hold, \textit{mutatis mutandis}, in the case of periodic boundary conditions, after suitable modification imposing a mean-free condition.
\end{remark}

\section{Preliminaries}\label{secPre}
\noindent


\noindent
In this section, we lay out some notation and preliminary notions that will be used later in the text. We use standard notation and results which can be found in, e.g., \cite{Constantin_Foias_1988, Temam_2001_Th_Num, Robinson_2001}. 
We assume that $\Omega \subset\nR^3$ is a bounded, open, connected set with $C^2$ boundary. 

Furthermore, we assume $\Omega$ is convex, so that there exists $c_1, c_2 > 0$ for which 
\begin{align}\label{Stokes_domain}
c_1 \normLp{2}{A \bu} \leq
\normHs{2}{\bu} \leq
c_2 \normLp{2}{A \bu},
\end{align}
where $A$ is defined in \eqref{Stokes} (see, e.g., \cite{dauge1989stationaryStokes, Guermond_Salgado_2011_Stokes}).
Let $C_c^\infty(\Omega)$ denote the space of smooth, compactly supported test functions from $\Omega$ to $\R^3$, and let $H_0^1(\Omega) \equiv H_0^1$ denote the closure of $C_c^\infty(\Omega)$ in $H^1(\Omega)$.

We set
\begin{align*}
    \mathcal{V} &= 
    \left\{\boldsymbol{\phi} \in C^\infty_c(\Omega): 
    \quad \grad\cdot \boldsymbol{\phi} = 0\right\},
\end{align*}
and let $H$ and $V$ be the closure of $\mathcal{V}$ in $L^2(\Omega)$ and $H^1(\Omega)$, respectively. 

We also denote the (real) $L^2$ inner-product and $H^m$ Sobolev norm by
\[
 (\bu,\bv) := \sum_{i=1}^3\int_{\Omega} u_i(\bx) v_i(\bx)\,d\bx,
\qquad
 \|\bu\|_{H^m} := 
 \lp \sum_{\abs{\alpha} \leq m} \normLp{2}{D^\alpha \bu}^2 \rp^\frac{1}{2},
\]
where $\alpha = \lp \alpha_1, \alpha_2, \alpha_3 \rp$ and $D^\alpha \bu = \partial_1^{\alpha_1}\partial_2^{\alpha_2}\partial_3^{\alpha_3} \bu$. For brevity, we will use the notation $L^2(\Omega) \equiv L^2$ and $H^m(\Omega) \equiv H^m$ throughout the paper.

We denote by $L^p(0,T; X)$ to space of Bochner integrable functions from $[0,T]$ to $X$ with norm given by 
\[ 
\| \bu \|_{L^p(0,T; X)} \equiv
\| \bu \|_{L^p X} := 
\lp\int_0^T \norm{\bu}_X^p \rp^{1/p}.  
\]

We recall Agmon's inequality for $s_1<\frac{n}{2}<s_2$, 
\begin{align}\label{agmon}
\|\bu\|_{L^\infty}\leq C\|\bu\|_{H^{s_1}}^{\theta}\| \bu \|_{H^{s_2}}^{1-\theta},
\end{align}
where $\frac{n}{2}=\theta s_1 + (1-\theta)s_2$, $\theta \in [0,1]$. We also recall the Gagliardo-Nirenberg-Sobolev interpolation inequality (see, e.g., \cite[p. 11]{Tao_2006}) 
in $\mathbb{R}^n$ for $1 \leq p,q < \infty$, 
\begin{align}\label{GNS}
  \normLp{p}{\bu} \leq C \normLp{q}{\bu}^\theta 
  \normLp{2}{D^\alpha\bu}^{1-\theta}, \qquad
  \frac{1}{p} = \frac{\theta}{q} + 
  \lp 1 -\theta \rp \lp \frac{1}{2} - \frac{\abs{\alpha}}{n} \rp.
\end{align}

Let $P_\sigma : L^2(\Omega) \to H$ be the Leray-Helmholtz orthogonal projection of $L^2(\Omega)$ onto $H$. Define the Stokes operator $A:\mathcal{D}(A) \subset H \to H$ as 
\begin{align}\label{Stokes}
    A := -P_\sigma \lap
\end{align}
with domain $\mathcal{D}(A) := H^2(\Omega) \cap V$. The operator $A$ is known to be positive-definite, self-adjoint, and with compact inverse $A^{-1}$ in $H$. From the Hilbert-Schmidt Theorem we obtain a sequence of eigenfunctions $\{\bw_j\}_{j=1}^\infty$ of $A^{-1}$, which are also eigenfunctions of $A$, with corresponding eigenvalues $\{\lambda_j\}_{j=1}^\infty$ such that $\{\bw_j\}_{j=1}^\infty$ is an orthonormal basis of $H$ and the sequence $\{\lambda_j\}_{j=1}^\infty$ is positive, monotone increasing, and tend toward infinity, so that 
$A\bw_j = \lambda_j \bw_j$ with $0 < \lambda_1 \leq \lambda_2 \leq \lambda_3 \leq \ldots$ and $\lim_{j\to\infty} \lambda_j = + \infty$. For further discussion see, e.g., \cite{Constantin_Foias_1988,Robinson_2001,Temam_2001_Th_Num}.
Due to the divergence-free condition, we can define a norm on $V$ by 
$$\left\langle A \bu, \bu \right\rangle
= \| A^{1/2}\bu \|_{L^2} 
= \normLp{2}{\grad\bu}^2.$$

We also recall the following Poincar\'e inequalities,
\begin{align*}\label{poincare}
    \normLp{2}{\bu}^2 &\leq \lambda_1^{-1} \normLp{2}{\grad\bu}^2 
    \quad \text{for all } \bu\in V, \\
    \normLp{2}{\grad \bu}^2 &\leq \lambda_1^{-1} \normLp{2}{A \bu}^2 
    \quad \text{for all } \bu \in \mathcal{D}(A).
\end{align*}
Denote by $P_m$ the projection onto the first $m$ eigenfunctions of $A$,
\begin{align}
    P_m \bu = \sum_{j=1 }^m u_j\bw_j.
\end{align}
This yields the following estimate: for $\bu \in H^s(\Omega),$ $s>0$, 
\begin{align}\label{proj_est}
    \normLp{2}{\lp I-P_m\rp \bu}^2 \leq \lambda_m^{-s} \normHs{s}{\bu}^2.
\end{align}

For $\bu \in C_c^\infty(\Omega)$ and $\bv \in \mathcal{V}$, we define the nonlinear term $B(\bu, \bv)$ by
\begin{subequations}\label{nonlinear_bi}
\begin{empheq}
{alignat=3}
\label{nonlinear_bi1}
B(\bu, \bv) & := P_\sigma \lp \lp \bu \cdot \grad \rp \bv\rp \hspace{6.5 mm}\quad \text{ in system \eqref{NSE}, or} \\ 
\label{nonlinear_bi2}
B(\bu, \bv) & := P_\sigma \lp \lp \grad \times \bv \rp \times \bu \rp \quad \text{ in system \eqref{NSE_vor}.}
\end{empheq}
\end{subequations}

For either case, the term $B(\cdot, \cdot)$ can be extended continuously to a bounded bilinear operator $B: H_0^1 \times V \to V'$. Similarly, we can define a trilinear operator $b: H_0^1\times V\times V \to \R$ by 
\begin{align}
    b(\bu, \bv, \bw) := \left \langle B(\bu, \bv), \bw \right\rangle
\end{align}
for all $\bu\in H_0^1$ and $\bv, \bw \in V$.

To recover the pressure term found in Systems \eqref{cNSE} and \eqref{cNSE_vor}, we will use a result of de Rham, which states for $\bbf \in C_c^\infty(\Omega)$, 
\begin{align}\label{deRham}
    \bbf = \grad p \text{ for some } p\in C_c^\infty(\Omega) \text{ if and only if } \left\langle \bbf, \bv \right\rangle = 0 \text{ for all } \bv\in \mathcal{V}.
\end{align}
See, e.g., \cite{Temam_2001_Th_Num,Wang_1993_deRham}.

\begin{lemma}\label{zeta_lem}
Suppose that $\bze$ is a calming function which satisfies condition \ref{zeta_cond_Lip} of Definition \ref{zeta_def}.
\begin{enumerate}
    \item If $\bu \in L^p(\Omega)$, $p \in [1,\infty]$, then $\bze(\bu)\in L^p(\Omega)$ and $\bze$ is Lipschitz in $ L^p(\Omega)$ with Lipschitz constant $C_L$.
    \item The mapping $I: L^2(0,T; H) \times L^2(0,T; V) \times L^2(0,T; H) \to \mathbb{R}$ defined by 
    \begin{align}
        I(\bu, \bv, \bw) = \int_0^T b(\bze(\bu), \bv,\bw) dt
    \end{align}
    is linear and continuous in its second and third component.
\end{enumerate}
\end{lemma}

\begin{proof}
    The proof can be found in \cite{Enlow_Larios_Wu_2023_calming} in the case 
    $$ b(\bze(\bu), \bv,\bw) = \lp P_\sigma( \lp\bze(\bu) \cdot \grad \rp \bv ), \bw \rp.$$ The proof for the case 
    $$ b(\bze(\bu), \bv,\bw) = \lp P_\sigma \lp \lp \grad\times \bv \rp \times \bze(\bu) \rp , \bw \rp $$ is nearly identical.
\end{proof}

 \section{Existence of Weak Solutions for Calmed NSE}\label{sec_cNSE_well_posedness}
We will prove the existence of solutions to \eqref{cNSE_weak} via Galerkin approximation. 
For $\bu_0 \in H$, the system
\begin{subequations}\label{gal_cNSE}
\begin{empheq}[left=\empheqlbrace]{alignat=3}
\dt \bu_m = - \nu A \bu_m - P_m B( \bze(\bu_m), \bu_m ) + P_m \bbf, \\ 
\bu_m(0, \bx) = P_m \bu_0(\bx)
\end{empheq}
\end{subequations}
is locally Lipschitz in $P_m(H)$ provided that $\bze$ is Lipschitz (see \cite{Enlow_Larios_Wu_2023_calming}, Lemma 3.2). So for each $m\in\nN$, there is some $T_m > 0$ for which a unique solution to \eqref{gal_cNSE} exists.

\subsection{Proof of Theorem \ref{thm_cNSE_vel_exi}}
Let $\bu_m$ be a solution to \eqref{gal_cNSE} on some maximum interval of existence $[0, T_m]$ with $T_m > 0$. 
Taking the inner product of \eqref{gal_cNSE} with $\bu^m$, we obtain
\begin{align*}
\frac{1}{2}\frac{d}{dt}\normLp{2}{\bu_m}^2 + 
\nu\normLp{2}{\grad\bu_m}^2 &= - 
\lp P_m B( \bze(\bu_m), \bu_m ), \bu_m  \rp + 
\left\langle  P_m \bbf, \bu_m \right\rangle \\ &= - 
\lp B( \bze(\bu_m), \bu_m ), \bu_m  \rp + 
\lp \bbf, \bu_m \rp. 
\end{align*}

Now, using H\"older's Inequality and Young's Inequality, 
\begin{align*}
&\quad
    \frac{1}{2}\frac{d}{dt}\normLp{2}{\bu_m}^2 + 
\nu\normLp{2}{\grad\bu_m}^2 
\\&\leq 
 \normLp{\infty}{\bze} \normLp{2}{\grad\bu_m} \normLp{2}{\bu_m} + 
 \norm{\bbf}_{V'}\normLp{2}{\grad \bu_m} \\ &\leq 
 \frac{\nu}{2}\normLp{2}{\grad\bu_m}^2 + 
 C_\nu\normLp{\infty}{\bze}^2 \normLp{2}{\bu_m}^2 + C_\nu \norm{\bbf}_{V'}^2.
\end{align*}

Rearranging terms yields the inequality
\begin{align}\label{gal_bound1}
    \frac{d}{dt}\normLp{2}{\bu_m}^2 + 
    \nu\normLp{2}{\grad\bu_m}^2 \leq 
    C_\nu \normLp{\infty}{\bze}^2\normLp{2}{\bu_m}^2 + 
    C_\nu \norm{\bbf}_{V'}^2.
\end{align}
Dropping the term $\nu\normLp{2}{\grad\bu_m}^2$ from the left-hand side of the inequality and applying Gr\"onwall's inequality yields
\begin{align}\label{cNSE_gal_bd1}
    \normLp{2}{\bu_m(t)}^2 & \leq 
    e^{ C_\nu \normLp{\infty}{\bze}^2 t} \normLp{2}{\bu_m(0)}^2 + 
    \int_0^t e^{ - C_\nu \normLp{\infty}{\bze}^2 \lp s-t\rp}
    \norm{\bbf}_{V'}^2 ds \\ &\leq 
    e^{ C_\nu \normLp{\infty}{\bze}^2 T_m} \lp \normLp{2}{\bu_0}^2 + 
    \| \bbf \|_{L^2 V'}^2 \rp. \notag
\end{align}
In fact, we can apply a standard bootstrapping argument to obtain that given any $T>0$, \eqref{cNSE_gal_bd1} remains valid if $T_m \equiv T$ for all $m\in\mathbb{N}$.
Thus $\bu_m$ is bounded in $L^\infty(0, T; L^2(\nT^2))$ independently of $m$. Integrating \eqref{gal_bound1} in time $t$ on the interval $[0,T]$, one obtains
\begin{align*}
    & \normLp{2}{\bu_m(T)}^2 - 
    \normLp{2}{\bu_m(0)}^2 + 
    \nu\int_0^T \normLp{2}{\grad\bu_m}^2 dt \\ &\leq 
    C_\nu \| \bbf \|_{L^2 V'}^2 + 
    C_\nu \normLp{\infty}{\bze}^2 
    \int_0^T \normLp{2}{\bu_m}^2 dt  \\ &\leq 
    C_\nu \| \bbf \|_{L^2 V'}^2 + 
    C_\nu \normLp{\infty}{\bze}^2
    T e^{ C_\nu \normLp{\infty}{\bze}^2 T} \lp \normLp{2}{\bu_0}^2 + 
    \| \bbf \|_{L^2 V'}^2 \rp.
 \end{align*}
Rearranging this inequality and applying \eqref{cNSE_gal_bd1} then yields, for a.e. $t\in [0,T]$,
\begin{align}
    \| \bu_m \|_{L^2 V}^2 \leq  
    C_\nu \lp 1 + \normLp{\infty}{\bze}^2
    T e^{ C_\nu \normLp{\infty}{\bze}^2 T} \rp 
    \lp \normLp{2}{\bu_0}^2 + 
    \| \bbf \|_{L^2 V'}^2 \rp
\end{align}
Therefore $\bu_m$ is bounded in $L^2(0, T; V)$ independently of $m$.

Now we check that $\dt \bu_m$ is bounded in $L^2(0, T; V')$ independently of $m$. Let $\bw \in V$ with $\normLp{2}{\grad\bw} = 1$. Taking the action of $\dt \bu$ on $\bw$  yields
\begin{align*}
    \abs{\left\langle \dt\bu_m, \bw \right\rangle} &\leq 
    \nu \abs{\left\langle A\bu_m, \bw \right\rangle} +
    \abs{\lp P_m B(\bze(\bu_m),\bu_m), \bw \rp} + 
    \abs{\lp P_m \bbf, \bw \rp} \\&=
    \nu \abs{\lp \grad\bu_m, \grad \bw \rp} +
    \abs{\lp B(\bze(\bu_m),\bu_m), P_m\bw \rp} + 
    \abs{\lp \bbf, P_m\bw \rp}.
\end{align*}
Note that 
\begin{align*}
\nu \abs{\lp \grad\bu_m, \grad\bw \rp} \leq 
\nu \normLp{2}{\grad\bu_m}\normLp{2}{\grad\bw} &=
\nu \normLp{2}{\grad\bu_m},
\end{align*}
and 
\begin{align*}
\abs{\lp B(\bze(\bu_m),\bu_m), P_m\bw \rp} \leq C
\normLp{\infty}{\bze} \normLp{2}{\grad\bu_m}\normLp{2}{\grad\bw} &=
C \normLp{\infty}{\bze} \normLp{2}{\grad\bu_m},
\end{align*}
and also
\begin{align*}
\abs{\lp \bbf, P_m\bw \rp} \leq
\norm{\bbf}_{V'}\norm{\bw}_V &\leq 
\norm{\bbf}_{V'}.
\end{align*}
From this we deduce that 
\begin{align}
    \norm{\dt \bu_m}_{L^2 V'} \leq C_{\nu, \epsilon} \lp \norm{\bu_m}_{L^2V} + \norm{\bbf}_{L^2V'} \  \rp, 
\end{align}
hence $\dt\bu_m$ is bounded in $L^2(0, T; V')$ independently of $m$.

By the Banach-Alaoglu Theorem and the above bounds, there exists $\bu \in L^\infty(0,T;H) \cap L^2(0,T; V)$ and a subsequence (which we will relabel as $\bu_m$) such that 
\begin{align}
    \label{BA_weakstar} \bu_m &\overset{\ast}{\rightharpoonup} \bu \text{ weak-$*$ in } L^\infty(0,T; H), \\
    \label{AL_weak}\bu_m &\rightharpoonup \bu \text{ weakly in } L^2(0,T; V), \\
    \label{AL_dt_weak}\dt\bu_m &\rightharpoonup \dt \bu \text{ weakly in } L^2(0,T; V').
\end{align}
Moreover, using the Aubin-Lions lemma one obtains another subsequence (still labelled as $\bu_m$) such that
\begin{align}
    \label{AL_strong} \bu_m &\to \bu \text{ strongly in } L^2(0,T; H).
\end{align}

Now we wish to show that passing to the limit in \eqref{gal_cNSE} yields $\eqref{cNSE_weak_eqn}$. 
Let $\bw \in V$, and set $\bv_m = \bu - \bu_m$.
We will show that $\bu$ is a solution to \eqref{cNSE_weak_eqn} by showing that 
\begin{align*}
    &\quad 
    \left\langle  \dt \bu, \bw \right\rangle +  
    \nu\lp \grad \bu, \grad\bw \rp + 
    b(\bze(\bu), \bu, \bw)  + \left\langle \bbf, \bw \right\rangle \\ &\quad- 
    \lp \dt \bu_m, \bw \rp - 
    \nu\lp\grad\bu_m, \bw \rp - 
    b\lp \bze(\bu_m), \bu_m, P_m\bw\rp - \lp P_m \bbf, \bw \rp 
\end{align*}
tends to $0$ as $m\rightarrow\infty$. This expression can be rewritten as follows.
  \begin{align*}
    &\quad   
    \left\langle  \dt \bv_m, \bw \right\rangle +  
    \nu\lp \grad \bv_m, \grad\bw \rp +
    b(\bze(\bu), \bv_m, \bw) \\ &\quad + 
    \lp 
    b(\bze(\bu),\bu_m, \bw) -
    b(\bze(\bu_m),\bu_m, \bw)
    \rp \\ &\quad + 
    b(\bze(\bu_m),\bu_m, \lp I-P_m\rp \bw) + 
    \left\langle (I - P_m)\bbf, \bw \right\rangle.
\end{align*}
Note that from \eqref{AL_weak} and \eqref{AL_dt_weak},
\[\lim_{m\to\infty} \int_0^T \left\langle \dt \bv_m, \bw \right\rangle + \nu\lp \grad \bv_m, \grad\bw \rp dt = 0 \] 
and 
\[
\lim_{m\to\infty} \int_0^T b(\bze(\bu), \bv_m, \bw) dt = 0
\]
by Lemma \ref{zeta_lem} and \eqref{AL_weak}. Now, using the Lipschitz property of $\bze$, H\"older's inequality, and the Gagliardo-Nirenberg-Sobolev inequality, we bound the fourth and fifth term as follows:
\begin{align}
    &\quad  \int_0^T \notag
    b(\bze(\bu),\bu_m, \bw) -
    b(\bze(\bu_m),\bu_m, \bw)
    dt \\ &\leq \notag
    \int_0^T
    \normLp{3}{\bv_m}
    \normLp{2}{\grad\bu_m}
    \normLp{6}{\bw} dt \\ &\leq \notag
    C \int_0^T
    \normLp{2}{\bv_m}^{\frac{1}{2}} 
    \normLp{2}{\grad \bv_m}^\frac{1}{2} 
    \normLp{2}{\grad\bu_m}
    \normLp{2}{\grad\bw}
    dt \\ &\leq 
    C \normLp{2}{\grad\bw}
    \| \bv_m \|^\frac{1}{2}_{L^2 L^2}
    \| \grad \bv_m \|^\frac{1}{2}_{L^2 L^2}
    \| \grad \bu_m \|_{L^2 L^2}.
    \notag
\end{align}
Therefore, since $\| \grad\bu_m \|_{L^2 L^2}$ and $\| \grad\bv_m \|_{L^2 L^2}$ are bounded and $\bv_m \to 0$ strongly,  
\begin{align}\notag
    \lim_{m \to \infty } \int_0^T \lp b(\bze(\bu),\bu_m, \bw) - b(\bze(\bu_m),\bu_m, \bw) \rp dt = 0
\end{align}
as a consequence of \eqref{AL_strong}. Finally, by \eqref{proj_est} we obtain 
\begin{align*}
    &\quad 
    \lim_{m\to\infty}\abs{\int_0^T b(\bze(\bu_m), \bu_m, (I-P_m)\bw) dt} \\ & \leq 
    \lim_{m\to \infty} \normLp{\infty}{\bze} \sup_{m\in\nN}{\norm{\bu_m}_{L^2V}}
    \lp  \lambda_m^{-1/2} \normHs{1}{\bw}\rp \\ &= 0
\end{align*}
and 
\begin{align*}
    \lim_{m\to\infty}\int_0^T \left\langle (I-P_m)\bbf, \bw \right\rangle dt = 0.
\end{align*}

Thus we deduce that a subsequence of solutions $\bu_m$ of \eqref{gal_cNSE} converges to a solution $\bu$ of \eqref{cNSE}. It remains to be shown that $\bu$ is continuous in time and satisfies the initial data. It is an immediate consequence of the Aubin-Lions Compactness Theorem (see, e.g., \cite[Corollary 7.3]{Robinson_2001}) that $\bu \in C([0,T]; L^2)$. To show that the initial data is satisfied, one carries out the procedure performed in, e.g., \cite{Enlow_Larios_Wu_2023_calming, Temam_2001_Th_Num}. \qedsymbol

\begin{remark}
    It is not known if weak solutions are unique for calmed NSE or calmed rNSE. Indeed, if $\bu_1$ and $\bu_2$ are weak solutions with same initial data $\bu_0$, one can write the difference equation $\eqref{Sys_uni}$ and obtain the energy equation \eqref{cNSE_uni1} as we do in the case of strong solutions. However, for weak solutions it does not seem possible to attain an upper bound for the term 
    $b (\bze(\bu_2) - \bze(\bu_1), \bu_2, \tu)$
    using the techniques seen in this paper.
\end{remark}

\section{Strong solutions}\label{sec_smooth_ini_data}

In this section we prove the first - and second - order regularity of weak solutions to the calmed NSE \eqref{cNSE_weak} for the purpose of showing that strong solutions are unique. To this end, we will apply the Aubin-Lions Compactness Theorem \cite[Corollary 7.3]{Robinson_2001}. 

\subsection{Proof of Theorem \ref{thm_cNSE_vel_reg}}
Here, we work formally, but the results can be made rigorous using the Galerkin procedure as in the proofs of the previous theorem. 
Suppose $\bu_0 \in V$ and $\bbf\in L^2(0,T;H)$ for some $T>0$. Taking the action of \eqref{cNSE_weak} with $A\bu$ and then using the Lions-Magenes Lemma, Young's inequality, and H\"older's inequality yields
\begin{align}
    \notag 
    & \quad \frac{1}{2}\frac{d}{dt}\normLp{2}{\grad \bu}^2 + 
    \nu \normLp{2}{A\bu}^2 =  
    b(\bze(\bu), \bu, A \bu)
    - \lp \bbf, A \bu \rp \\ \notag & \leq 
    \normLp{\infty}{\bze} \normLp{2}{\grad\bu} \normLp{2}{A\bu} + 
    \normLp{2}{\bbf}\normLp{2}{A\bu}\\ \notag & \leq 
    C_\nu \normLp{\infty}{\bze}^2 \normLp{2}{\grad\bu}^2 + 
    C_\nu \normLp{2}{\bbf}^2 +
    \frac{\nu}{2}\normLp{2}{A\bu}^2 
\end{align}
    Rearranging these terms yields the inequality
\begin{align}\label{cNSE_Hor1}
    \frac{d}{dt}\normLp{2}{\grad\bu}^2 + 
    \nu \normLp{2}{A\bu}^2 \leq 
    C_\nu \normLp{\infty}{\bze}^2 \normLp{2}{\grad\bu}^2 + 
    C_\nu \normLp{2}{\bbf}^2.
\end{align}
We now remove the viscosity term and apply Gr\"onwall's inequality to obtain for a.e. $t\in[0,T]$,
\begin{align} \label{cNSE_Greg0}
    \normLp{2}{\grad\bu(t)}^2 \leq 
    e^{C_\nu \normLp{\infty}{\bze}^2 t} \normLp{2}{\grad\bu_0}^2 +
    C_\nu \int_0^t e^{- C_\nu \normLp{\infty}{\bze}^2\lp s-t\rp} \normLp{2}{\bbf(s)}^2\,ds.
\end{align}
Thus $\bu \in L^\infty(0,T; V)$ whenever $\bu_0 \in V$ and $\bbf \in L^2(0,T;H)$.
Returning to \eqref{cNSE_Hor1}, we integrate in time to obtain 

\begin{align} \label{cNSE_Greg0p5}
    \nu\int_0^T \normLp{2}{A\bu}^2 dt 
    \leq 
    \normLp{2}{\grad\bu_0}^2 + 
    C_\nu \int_0^T \normLp{\infty}{\bze}^2
    \normLp{2}{\grad\bu}^2 + \normLp{2}{\bbf}^2 dt.
\end{align}
From estimates \eqref{cNSE_Greg0} and \eqref{cNSE_Greg0p5} we deduce that $\bu \in L^2(0,T; H^2\cap V)$.
It remains to be shown that $\dt\bu \in L^2(0,T; H)$. This follows immediately from the calculation below:

\begin{align*}
\int_0^T \normLp{2}{\dt\bu}^2 dt &=
\int_0^T \normLp{2}{\nu A \bu + B(\bze(\bu), \bu) + \bbf }^2 dt \\ &\leq
C\int_0^T \nu \normLp{2}{A\bu}^2 + 
\normLp{\infty}{\bze}^2\normLp{2}{\grad\bu}^2 + \normLp{2}{\bbf}^2
dt \\ &< \infty.
\end{align*}
Therefore $\dt\bu \in L^2(0,T; H)$. By the Aubin-Lions Compactness Theorem, we conclude that $\bu \in C([0,T]; V)$. 
\qedsymbol \vspace{3 mm}\\

We now proceed in showing the global existence and uniqueness of strong solutions. With the existence of such solutions already known from prior results, this theorem focuses on uniqueness and continuous dependence on initial data.
\subsection{Proof of Theorem \ref{thm_cNSE_vel_uni}}
From Theorems \ref{thm_cNSE_vel_exi}, \ref{thm_cNSE_vel_reg}, and from \eqref{deRham}, we deduce the existence of strong solutions to calmed NSE \eqref{cNSE} and calmed rNSE \eqref{cNSE_vor} satisfying the hypotheses of Definition \ref{cNSE_vel_strong_soln}. Suppose $\bu_1$ and $\bu_2$ are strong solutions with respective initial data $\bu_0^1, \bu_0^2 \in V$ and forcing term $\bbf \in L^2(0,T;H)$. Let $\tu = \bu_1 - \bu_2$ and $\tu_0  = \bu_0^1 - \bu_0^2$. When we take the difference of the two equations we obtain
\begin{align} \label{Sys_uni}
    \pd{\tu}{t} - \nu \lap \tu =  
    B( \bze(\bu_2) - \bze(\bu_1), \bu_2 ) - 
    B(\bze(\bu_1), \tu ).
\end{align}
We now take the inner-product with $\tu$, which yields 
\begin{align}\label{cNSE_uni1}
&\quad
    \frac{1}{2}\frac{d}{dt}\normLp{2}{\tu}^2 + 
    \nu \normLp{2}{\grad \tu}^2 
    \\&= \notag
    b( \bze(\bu_2) - \bze(\bu_1), \bu_2, \tu ) - 
    b(\bze(\bu_1), \tu, \tu).
\end{align}
For the first term, using H\"older's inequality, the Gagliardo-Nirenberg-Sobolev inequality, condition \ref{zeta_cond_Lip} of Definition \ref{zeta_def}, and Poincar\`e's inequality, one obtains
\begin{align*}
    &\quad \abs{b( \bze(\bu_2) - \bze(\bu_1), \bu_2, \tu )} \\ &\leq 
    \normLp{3}{\tu}\normLp{6}{\grad \bu_2}\normLp{2}{\tu} \\ &\leq 
    C \normLp{2}{\tu}^\frac{3}{2} \normLp{2}{\grad\tu}^\frac{1}{2}\normLp{2}{\lap\bu_2} \\ &\leq 
    C_\nu \normLp{2}{\lap\bu_2}^\frac{4}{3}  
    \normLp{2}{\tu}^2 + 
    \frac{\nu}{4}\normLp{2}{\grad\tu}^2.
\end{align*}
While for the second term, one obtains
\begin{align*}
    \abs{b(\bze(\bu_1), \tu, \tu)} \leq 
    C_\nu \normLp{\infty}{\bze}^2 \normLp{2}{\tu}^2 + 
    \frac{\nu}{4}\normLp{2}{\grad\tu}^2.
\end{align*}
Inserting these bounds into estimate \eqref{cNSE_uni1} then yields 
\begin{align}
    \frac{d}{dt}\normLp{2}{\tu}^2 + 
    \nu \normLp{2}{\grad \tu}^2 \leq 
    C_\nu \lp 
    \normLp{\infty}{\bze}^2 + \normLp{2}{\lap \bu_2}^\frac{4}{3}
    \rp \normLp{2}{\tu}^2.
\end{align}
Since $\bu_2$ is a strong solution to calmed NSE \eqref{cNSE} we have the containment $\bu_2 \in L^2(0,T; H^2\cap V)$, hence
\[
A(T) := C_\nu \int_0^T \lp \normLp{\infty}{\bze}^2 + \normLp{2}{\lap \bu_2}^\frac{4}{3} \rp dt < \infty.
\]
Using Gr\"onwall's inequality, it follows that,
\begin{align}\label{cNSE_uni2}
    \normLp{2}{\tu(t)}^2 \leq e^{A(T)}\normLp{2}{\tu_0}^2.
\end{align}
We conclude that strong solutions to \eqref{cNSE} are unique and depend continuously on initial data. \qedsymbol

\section{Convergence to strong solutions of the Navier-Stokes equations}\label{sec_convergence}
In this section we prove that strong solutions $\bu^\epsilon$ to calmed NSE will converge to a strong solution $\bu$ to NSE on sufficiently small time intervals when $\bze$ is known to satisfy condition \ref{zeta_cond_conv} for some minimal value $\beta \geq 1$. 

\subsection{Proof of Theorem \ref{thm_cNSE_conv}}
    Set $\bw^\epsilon = \bu - \bu^\epsilon$. We then take the action of the difference of \eqref{NSE} and \eqref{cNSE} with $A\bw^\epsilon$ and use the Lions-Magenes Lemma to obtain 
    \begin{align} \label{conv_diffEQ}
        \frac{1}{2}\frac{d}{dt}\normLp{2}{\grad\bw^\epsilon}^2 + 
        \nu \normLp{2}{A\bw^\epsilon}^2 = N,
    \end{align}
    where the nonlinearity $N$ is rewritten as 
        \begin{align*}
            N &= \quad 
            b(\bze(\bu) - \bu, \bu, A \bw^\epsilon) -            
            b(\bze(\bu), \bw^\epsilon, A \bw^\epsilon)
            \\ &\quad -
            b(\bze(\bu) - \bze(\bu^\epsilon), \bu, A \bw^\epsilon)
            +
            b(\bze(\bu) -\bze(\bu^\epsilon), \bw^\epsilon, A \bw^\epsilon )
            \\ &= \quad N_1 + N_2 +N_3 + N_4.
    \end{align*}

    For $N_1$, we use condition \ref{zeta_cond_conv} of Definition \ref{zeta_def}, Agmon's inequality, and Young's inequality to obtain
    \begin{align} \label{n1}
        \abs{N_1} 
        &\leq \int_\Omega \abs{\bze(\bu) - \bu} 
        \abs{\grad\bu}\abs{A\bw^\epsilon}\, d\bx \\ &\leq \notag
        C\epsilon^\alpha \int_\Omega \abs{\bu}^\beta \abs{\grad\bu}\abs{A \bw^\epsilon}\, d\bx \\ &\leq \notag
         C\epsilon^\alpha \normLp{\infty}{\bu}^\beta 
         \normLp{2}{\grad \bu} \normLp{2}{A\bw^\epsilon} \\ &\leq \notag
         C\epsilon^\alpha 
         \|\bu\|_{L^\infty V} 
          \normLp{2}{\lap\bu}^\beta
         \normLp{2}{A\bw^\epsilon} \\ &\leq \notag
         C_\nu \| \bu\|_{L^\infty V}^2 \normLp{2}{\lap\bu}^{2\beta} \epsilon^{2\alpha} + 
         \frac{\nu}{8}\normLp{2}{A\bw^\epsilon}^2.
    \end{align}
    For the remaining terms, we use a combination of Agmon's inequality, Poincare's inequality, H\"older's inequality, the Gagliardo-Nirenberg-Sobolev inequality, and Young's inequality. For $N_2$, we obtain
    \begin{align}\label{n2}
        \abs{N_2} &\leq \int_\Omega \abs{\bu}\abs{\grad\bw^\epsilon}\abs{A\bw^\epsilon}\, d\bx \\ &\leq \notag
        \normLp{\infty}{\bu}\normLp{2}{\grad\bw^\epsilon}\normLp{2}{A\bw^\epsilon} \\ &\leq \notag
        C \normLp{2}{\lap\bu}\normLp{2}{\grad\bw^\epsilon}\normLp{2}{A\bw^\epsilon} \\ &\leq \notag
        C_{\nu}\normLp{2}{\lap\bu}^2 \normLp{2}{\grad\bw^\epsilon}^2 + 
        \frac{\nu}{8}\normLp{2}{A\bw^\epsilon}^2,
    \end{align}
    where we use the additional fact that $\abs{\bze(\bu)}\leq \abs{\bu}$, which follows from conditions \ref{zeta_cond_Lip} and \ref{zeta_cond_conv} of Definition \ref{zeta_def}. For $N_3$,
    \begin{align}\label{n3}
        \abs{N_3} &\leq \int_\Omega \abs{\bw^\epsilon}\abs{\grad\bu}\abs{A\bw^\epsilon}\, d\bx \\ &\leq \notag
        \normLp{3}{\grad\bu}\normLp{6}{\bw^\epsilon}\normLp{2}{A\bw^\epsilon} \\ &\leq \notag
        C \normLp{2}{\grad\bu}^\frac{1}{2}\normLp{2}{\lap\bu}^\frac{1}{2}\normLp{2}{\grad\bw^\epsilon}\normLp{2}{A\bw^\epsilon} \\ &\leq \notag
        C_\nu \| \bu \|_{L^\infty V} \normLp{2}{\lap \bu} 
        \normLp{2}{\grad\bw^\epsilon}^2 + \frac{\nu}{8}\normLp{2}{A\bw^\epsilon}^2
    \end{align}
    and for $N_4$, using inequality \eqref{Stokes_domain} we deduce
    \begin{align}\label{n4}
        \abs{N_4} &\leq \int_\Omega \abs{\bw^\epsilon}\abs{\grad\bw^\epsilon}\abs{A\bw^\epsilon}\, d\bx \\ &\leq \notag
        \normLp{6}{\bw^\epsilon}\normLp{3}{\grad\bw^\epsilon}\normLp{2}{A\bw^\epsilon} \\ &\leq \notag
        \normLp{2}{\grad\bw^\epsilon}^\frac{3}{2}\normHs{2}{\bw^\epsilon}^\frac{1}{2}\normLp{2}{A\bw^\epsilon}\\ &\leq \notag
        C\normLp{2}{\grad\bw^\epsilon}^\frac{3}{2}\normLp{2}{A\bw^\epsilon}^\frac{3}{2} \\ &\leq \notag
        C_\nu \normLp{2}{\grad\bw^\epsilon}^6 + 
        \frac{\nu}{8}\normLp{2}{A\bw^\epsilon}^2.
    \end{align}
    We now make the ansatz 
    \begin{align} \label{conv_ansatz}
        \normLp{2}{\grad\bw^\epsilon} < 1,
    \end{align}
    which holds at the initial time by assumption and therefore for a short time since $\bu, \bu^\epsilon \in C([0,T];V)$. We want to show that this leads to an even tighter bound. To this end, we apply \eqref{conv_ansatz} to estimate \eqref{n4}, then insert the bounds \eqref{n1}, \eqref{n2}, \eqref{n3}, and \eqref{n4} into estimate \eqref{conv_diffEQ} which yields
    \begin{align} \label{cNSE_conv2}
        &\quad \frac{d}{dt}\normLp{2}{\grad\bw^\epsilon}^2 + 
        \nu \normLp{2}{ A \bw^\epsilon}^2 \\ &\leq \notag
        C_\nu \|\bu\|_{L^\infty V}^2 \normLp{2}{\lap\bu}^{2\beta} \epsilon^{2\alpha} + 
        C_\nu 
        \lp 
        \normLp{2}{\lap\bu}^2 +
        \|\bu\|_{L^\infty V}
        \normLp{2}{\lap\bu} + 1
        \rp \normLp{2}{\grad\bw^\epsilon}^2.
    \end{align}
    By \eqref{NSE_reg_cond} we deduce that the first term and the factor preceding $\normLp{2}{\grad\bw^\epsilon}^2$ in \eqref{cNSE_conv2} are integrable in time.
    Since $\normLp{2}{\grad\bw^\epsilon(0)} = 0$, we can use Gr\"onwall's inequality to obtain, for all $t\in [0,T]$,
    \begin{align}\label{cNSE_conv3}
        \normLp{2}{\grad\bw^\epsilon(t)} \leq K \epsilon^\alpha,
    \end{align}
    where $K > 0$ is a constant depending on 
    $\Omega, \nu, \beta, \| \bu \|_{L^\infty V}, \normLp{2}{\lap\bu}$, and $T$. By taking $\epsilon > 0$ sufficiently small, it follows that 
    \[
    \normLp{2}{\grad\bw^\epsilon(t)} < \frac{1}{2}
    \]
    for all $t\in [0,T]$.
    After applying a standard bootstrapping argument (see, e.g., \cite{Enlow_Larios_Wu_2023_calming}), we conclude that inequality \eqref{cNSE_conv3} is valid for all $t\in [0,T]$. \qedsymbol

\section{Existence of Strong Solutions to 3D Navier-Stokes}\label{sec_crNSE_cauchy}
To prove Theorem \ref{thm_cNSE_Cauchy}, we begin with a lemma establishing higher-order bounds that are independent of the calming parameter.  We assume a uniform-in-time bound on $\bbf$, namely $\bbf\in L^\infty((0,\infty);L^2)$. This hypothesis could likely be weakened, but simplicity of presentation, we do not pursue this here.
\begin{lemma}\label{lem_grad_bd}
    Let $\nu>0$.  Suppose, for each $\epsilon > 0$, $\bu^\epsilon$ is a strong solution to calmed NSE \eqref{cNSE} or calmed rNSE \eqref{cNSE_vor} with initial data $\bu_0 \in V$ and $\bbf\in L^\infty((0,\infty);L^2)$. On the interval $[0, T_0]$, where 
    \begin{align}\label{lem_time}
        T_0 := \frac{\lp \normLp{2}{\grad\bu_0}^2 + M^2 \rp^{-2} - 
        \frac{1}{4}\normLp{2}{\grad\bu_0}^{-4}} {C_\nu}
    \end{align}
    and $M:= \|\bbf\|_{L^\infty((0,\infty);L^2)}^\frac{1}{3}$,
    the following inequalities are valid:
    \begin{align}\label{lem_ineq1}
        \sup_{\epsilon >0} \sup_{t\in [0,T_0]} \normLp{2}{\grad\bu^\epsilon(t)}^2 \leq 2\normLp{2}{\grad\bu_0}^2
    \end{align}
    and 
    \begin{align}\label{lem_ineq2}
        \sup_{\epsilon >0}\left\{\nu \int_0^{T_0} \normLp{2}{A\bu^\epsilon}^2 \right\} \leq 
        \normLp{2}{\grad\bu_0}^2 + 
        C_\nu 
        T_0 
        \lp
        2
        \normLp{2}{\grad\bu_0}^2
        + 
        M^2
        \rp^3,
    \end{align}
    where $C_\nu$ is a positive constant that depends on the domain $\Omega$ and $\nu$ and may change from line to line.
\end{lemma}

\begin{proof}
    Following similar steps as before in showing higher-order regularity, we take the action of \eqref{cNSE_vor} on $A\bu^\epsilon$, integrate by parts, and apply the Lions-Magenes Lemma, to obtain 
    \begin{align}\notag
        \frac{1}{2}\frac{d}{dt}\normLp{2}{\grad\bu^\epsilon}^2 + \nu \normLp{2}{A\bu^\epsilon}^2 = 
        b(\bu^\epsilon, \bu^\epsilon, A \bu^\epsilon) + \lp \bbf, A\bu^\epsilon \rp 
    \end{align}
    Now we use the Gagliardo-Nirenberg-Sobolev inequality, the Cauchy-Schwarz inequality, \eqref{Stokes_domain}, and Young's inequality, which yields
    \begin{align}
        &\quad \notag
        \frac{1}{2}\frac{d}{dt}\normLp{2}{\grad\bu^\epsilon}^2 + 
        \nu \normLp{2}{A \bu^\epsilon}^2 \\ &\leq \notag
        \normLp{3}{\grad \bu^\epsilon}
        \normLp{6}{\bu^\epsilon}
        \normLp{2}{A \bu} + 
        \normLp{2}{\bbf}\normLp{2}{A\bu^\epsilon}
        \\ &\leq \notag 
        C \normLp{2}{\grad \bu^\epsilon}^\frac{3}{2}
        \normHs{2}{\bu^\epsilon}^\frac{1}{2}
        \normLp{2}{A \bu^\epsilon} + 
        \normLp{2}{\bbf}\normLp{2}{A\bu^\epsilon}
        \\ &\leq \notag 
        C \normLp{2}{\grad \bu^\epsilon}^\frac{3}{2}
        \normLp{2}{A \bu^\epsilon}^\frac{3}{2} +
        \normLp{2}{\bbf}\normLp{2}{A\bu^\epsilon} \\ &\leq \notag 
        C_{\nu} \normLp{2}{\grad\bu^\epsilon}^6 + 
        C_\nu \normLp{2}{\bbf}^2 + 
        \frac{\nu}{2}\normLp{2}{A\bu^\epsilon}^2 \\ &\leq \notag 
        C_{\nu} \normLp{2}{\grad\bu^\epsilon}^6 + 
        C_\nu M^6 + 
        \frac{\nu}{2}\normLp{2}{A\bu^\epsilon}^2 \\ &\leq \notag
        C_\nu 
        \lp \normLp{2}{\grad\bu^\epsilon}^2 + 
        M^2 \rp^3 + 
        \frac{\nu}{2}\normLp{2}{A\bu^\epsilon}^2.
    \end{align}
    We now rewrite this inequality as 
    \begin{align}\label{lem_grad_bd_1}
        \frac{d}{dt}\normLp{2}{\grad\bu^\epsilon}^2 + 
        \nu\normLp{2}{A\bu^\epsilon}^2 \leq C_\nu 
        \lp \normLp{2}{\grad\bu^\epsilon}^2 + 
        M^2 \rp^3
    \end{align}
    which, after making the substitution $\eta  = \normLp{2}{\grad\bu^\epsilon}^2 + M^2$ and removing the diffusive terms, becomes
    \begin{align}\notag
        \frac{d}{dt}\eta \leq C_\nu \eta^3. 
    \end{align}
    From this inequality we derive, for all $t\in[0, T_0]$,
    \begin{align}\notag
        \eta(t) \leq 
        \lp 
        \eta(0)^{-2} - C_\nu T_0
        \rp^{- \frac{1}{2}}
    \end{align}
    or 
    \begin{align}\label{lem_grad_bd_2}
        \normLp{2}{\grad\bu^\epsilon}^2 + 
        M^2 \leq 
        \lp 
        \lp
        \normLp{2}{\grad\bu_0}^2 + 
        M^2 
        \rp^{-2} - 
        C_\nu T_0
        \rp^{- \frac{1}{2}} = 
        2\normLp{2}{\grad\bu_0}^2
    \end{align}
    for $T_0$ as in \eqref{lem_time}, thus proving \eqref{lem_ineq1}.
    We now return to estimate \eqref{lem_grad_bd_1}, integrate in time on the interval $[0,T_0]$, and apply estimate \eqref{lem_grad_bd_2} to obtain
\begin{align}
        &\quad \notag
        \nu \int_0^{T_0} \normLp{2}{A \bu^\epsilon}^2 dt \\ &\leq \notag
        \normLp{2}{\grad\bu_0}^2 +
         C_\nu \int_0^{T_0} 
         \lp 
        \normLp{2}{\grad\bu^\epsilon}^2 + M^2 \rp^3
        dt \\ &\leq \notag
        \normLp{2}{\grad\bu_0}^2 + 
        C_\nu T_0 \lp 
        2\normLp{2}{\grad\bu_0}^2 + M^2 \rp^3.
\end{align}
This proves \eqref{lem_ineq2}.
\end{proof}

Our lemma guarantees that for nonzero initial data $\bu_0 \in V$ and forcing term $\bbf\in L^\infty(0,\infty; L^2)$, there exists a positive time $T_0$ for which, given any $\epsilon > 0$, $\bu^\epsilon$ is bounded in $L^\infty(0,T_0; V) \cap L^2(0, T_0; H^2\cap V)$. We now, proceed to show that, on the time interval $[0,T_0]$, $\left\{ \bu^\epsilon \right\}_{\epsilon >0}$ is Cauchy. 

\subsection{Proof of Theorem \ref{thm_cNSE_Cauchy}}
Let $\bu^\epsilon$ and $\bu^\delta$ be strong solutions to calmed NSE \eqref{cNSE} or calmed rNSE \eqref{cNSE_vor} with initial data $\bu_0 \in V$ and with respective calming parameters $\epsilon >0$ and $\delta > 0$. From the results of Lemma \ref{lem_grad_bd} we ascertain the existence of a maximal time $T > 0$ for which
\begin{align}\label{cauchy1}
    \sup_{\epsilon >0} \sup_{t\in [0,T]} \normLp{2}{\grad\bu^\epsilon(t)}^2 \leq 2 \normLp{2}{\grad\bu_0}^2.
\end{align}
From Lemma \ref{lem_grad_bd} we ascertain that $T \geq \frac{1}{4}\normLp{2}{\grad\bu_0}^{-4} > 0$.
Set $\tu = \bu^\epsilon - \bu^\delta$. The system for $\tu$ can be written as  
\begin{align} \notag
&\quad
    \dt \tu + \nu A\tu  
    \\&=  - \notag
    B(\bzd(\bu^\delta), \tu) -
    B(\tu, \bu^\epsilon) +
    B(\bzd(\bu^\delta) - \bu^\delta,\bu^\epsilon) +
    B(\bu^\epsilon - \bze(\bu^\epsilon), \bu^\epsilon)
\end{align}
We then take the inner product with $\tu$ to obtain 
\begin{align}\notag
    \frac{1}{2}\frac{d}{dt}\normLp{2}{\tu}^2 + 
    \nu \normLp{2}{\grad\tu}^2 &\leq  
    \int_\Omega \abs{\grad\tu}\abs{\bzd(\bu^\delta)}\abs{\tu}\, d\bx 
    \\ &+ \notag
    \int_\Omega \abs{\tu}^2 \abs{\grad\bu^\epsilon} \, d\bx \\ &+ \notag
    \int_\Omega \abs{\grad\bu^\epsilon}\abs{\bzd(\bu^\delta) - \bu^\delta}\abs{\tu}\, d\bx \\ &+ \notag 
    \int_\Omega \abs{\grad\bu^\epsilon}\abs{\bze(\bu^\epsilon) - \bu^\epsilon}\abs{\tu}\, d\bx.
\end{align}
Now, applying condition \ref{zeta_cond_conv} of $\bze$ yields 
\begin{align} \notag
    \frac{1}{2}\frac{d}{dt}\normLp{2}{\tu}^2 + 
    \nu \normLp{2}{\grad\tu}^2 \leq  
    \int_\Omega \abs{\grad\tu}\abs{\bu^\delta}\abs{\tu}\, d\bx +
    \int_\Omega \abs{\tu}^2 \abs{\grad\bu^\epsilon} \, d\bx 
    \\ \qquad + \notag 
    C\delta^\alpha \int_\Omega \abs{\grad\bu^\epsilon}\abs{\bu^\delta}^\beta\abs{\tu}\, d\bx +
    C\epsilon^\alpha\int_\Omega \abs{\grad\bu^\epsilon}\abs{\bu^\epsilon}^\beta\abs{\tu}\, d\bx.
\end{align}
Using H\"older's inequality, Agmon's inequality, Poincar\'e's inequality, \eqref{Stokes_domain}, and Young's inequality, for the first term we obtain 
\begin{align}\label{cauchy2}
    \int_\Omega \abs{\grad\tu}\abs{\bzd(\bu^\delta)}\abs{\tu}\, d\bx &\leq 
    \normLp{2}{\grad\tu} \normLp{\infty}{\bu^\delta}\normLp{2}{\tu} \\ &\leq \notag
    C\normLp{2}{\grad\tu} \normHs{2}{\bu^\delta}\normLp{2}{\tu} \\ &\leq \notag
    C\normLp{2}{\grad\tu} \normLp{2}{A\bu^\delta}\normLp{2}{\tu} \\ &\leq \notag
    C_\nu \normLp{2}{A \bu^\delta}^2 \normLp{2}{\tu}^2 + \frac{\nu}{8}\normLp{2}{\grad\tu}^2.
\end{align}
and similarly for the second term, using also \eqref{lem_ineq1},
\begin{align}\label{cauchy2p5}
    \int_\Omega \abs{\tu}^2 \abs{\grad\bu^\epsilon} \, d\bx
    &\leq 
    \normLp{3}{\tu}
    \normLp{2}{\grad\bu^\epsilon}
    \normLp{6}{\tu} 
    \\ &\leq \notag  
    C\normLp{2}{\tu}^\frac{1}{2}
    \normLp{2}{\grad\bu_0}
    \normLp{2}{\grad\tu}^\frac{3}{2} 
    \\ &\leq \notag
    C_{\nu}\normLp{2}{\grad\bu_0}^4
    \normLp{2}{\tu}^2 + 
    \frac{\nu}{8}
    \normLp{2}{\grad\tu}^2
\end{align}
Using the same inequalities for the third term, we deduce that 
\begin{align}\notag
    C\delta^\alpha \int_\Omega \abs{\grad\bu^\epsilon}\abs{\bu^\delta}^\beta\abs{\tu}\, d\bx 
    &\leq 
    C\delta^\alpha 
    \normLp{6}{\grad\bu^\epsilon}
    \normLp{2\beta}{\bu^\delta}^\beta
    \normLp{3}{\tu} \\ &\leq \notag
    C\delta^\alpha 
    \normHs{2}{\bu^\epsilon}
    \normLp{2\beta}{\bu^\delta}^\beta
    \normLp{2}{\tu}^\frac{1}{2} 
    \normLp{2}{\grad\tu}^\frac{1}{2} \\ &\leq \notag
    C\delta^\alpha 
    \normLp{2}{A\bu^\epsilon}
    \normLp{2\beta}{\bu^\delta}^\beta
    \normLp{2}{\tu}^\frac{1}{2} 
    \normLp{2}{\grad\tu}^\frac{1}{2}.
\end{align}
For $\beta \in [1,3]$, from the Gagliardo-Nirenberg-Sobolev inequality and \eqref{cauchy1} we have
\begin{align}\notag
    \normLp{2\beta}{\bu^\delta}^\beta \leq C_\beta \normLp{2}{\grad\bu_0}^\beta.
\end{align}
We insert this bound into estimate \eqref{cauchy2} and apply Young's inequality to obtain
\begin{align}\label{cauchy3}
    &\quad C\delta^\alpha \int_\Omega \abs{\grad\bu^\epsilon}\abs{\bu^\delta}^\beta\abs{\tu}\, d\bx 
    \\ &\leq \notag
    C_{\beta, \nu}\delta^{2\alpha}\normLp{2}{\grad\bu_0}^{2\beta}\normLp{2}{A\bu^\epsilon}^2 + 
    C_\nu \normLp{2}{\tu}^2 + \frac{\nu}{8}\normLp{2}{\grad\tu}^2.
\end{align}
We follow the same procedure for the final term:
\begin{align}\label{cauchy4}
    &\quad C\epsilon^\alpha \int_\Omega \abs{\grad\bu^\epsilon}\abs{\bu^\epsilon}^\beta\abs{\tu}\, d\bx 
    \\ &\leq \notag
    C_{\beta, \nu}\epsilon^{2\alpha}\normLp{2}{\grad\bu_0}^{2\beta}\normLp{2}{A\bu^\epsilon}^2 + 
    C_\nu \normLp{2}{\tu}^2 + \frac{\nu}{8}\normLp{2}{\grad\tu}^2.
\end{align}
Invoking \eqref{cauchy2}, \eqref{cauchy2p5}, \eqref{cauchy3}, and \eqref{cauchy4} yields the upper bound
\begin{align}\label{cauchy5}
    &\quad \frac{d}{dt}\normLp{2}{\tu}^2 + 
    \normLp{2}{\grad\tu}^2 
    \\ &\leq \notag
    C_\nu \lp 
    \normLp{2}{A\bu^\delta}^2 + 
    \normLp{2}{\grad\bu_0}^4
    + 1 \rp \normLp{2}{\tu}^2 + 
    C_{\beta,\nu} 
    \normLp{2}{\grad\bu_0}^{2\beta}
    \normLp{2}{A\bu^\epsilon}^2
    \lp 
    \delta^{2\alpha} + \epsilon^{2\alpha}
    \rp \\ &\leq \notag 
    K_1\normLp{2}{\tu}^2 + 
    K_2 \lp 
    \delta^{2\alpha} + \epsilon^{2\alpha}
    \rp,
\end{align}
where $K_1$ and $K_2$ depend on $\nu, \beta, T$ and $\normLp{2}{\grad\bu_0}$, but not $\epsilon$ or $\delta$, and are determined by Lemma \ref{lem_grad_bd}. Now, we apply Gr\"onwall's inequality to obtain, for all $t \in [0,T]$,
\begin{align}
    \normLp{2}{\tu(t)} \leq 
    K_3\lp \delta^{2\alpha} + \epsilon^{2\alpha}\rp,
\end{align}
where 
\[
K_3 = \frac{K_2}{K_1}\lp e^{K_1 T} - 1\rp
\lp \delta^{2\alpha} + \epsilon^{2\alpha}\rp.
\]
Therefore we see that $\displaystyle\lim_{\delta, \epsilon \to 0} \normLp{2}{\bu^\epsilon - \bu^\delta} = 0$, hence $\{\bu^\epsilon\}_{\epsilon > 0}$ is Cauchy in $L^\infty H$ with respect to the calming parameter. If instead we integrate \eqref{cauchy5} on $[0,T]$, we can derive the upper bound 
\begin{align}
     \nu \int_0^T \normLp{2}{\grad\tu}^2 dt  & \leq 
    K_1 \int_0^T \normLp{2}{\tu}^2 dt + 
    K_2 T \lp \delta^{2\alpha} + \epsilon^{2\alpha}\rp \\ & \leq \notag 
    K_1 T \| \tu  \|_{L^\infty L^2} +
    K_2 T \lp \delta^{2\alpha} + \epsilon^{2\alpha}\rp ,
\end{align}
hence $\{\bu^\epsilon\}_{\epsilon > 0}$ is also Cauchy in $L^2 V$. 
Therefore, there exists $\bu \in L^\infty H \cap L^2 V$ for which 
\begin{align}\label{strong_eps_conv}
\bu^\epsilon \to \bu \text{ strongly in }\bu \in L^\infty H \cap L^2 V
\end{align}
 as $\epsilon \to 0$. We now show that this limit point $\bu$ is in fact a solution to 3D rNSE \eqref{NSE_vor}. First note that, owing to Lemma \ref{lem_grad_bd}, the equivalence \eqref{Stokes_domain}, the Banach-Alaoglu Theorem, and the usual uniqueness of limits, it follows that
\begin{align}\label{NSE_target_strong_spaces}
    \bu\in L^\infty V \cap L^2 (H^2\cap V).
\end{align}
 Set $\bu^* = \bu^\epsilon - \bu$, and take the action of \eqref{NSE_vor} against an arbitrary test function $\bw\in C^1_c([0,T);V)$ and integrate by parts in time (noting that $\bw\big|_{t=T}=\mathbf{0}$),
\begin{align*}
    &\quad 
    -\int_0^T\left\langle  \bu^\epsilon, \dt \bw \right\rangle\,dt +  
    \nu\int_0^T\lp \grad \bu^\epsilon, \grad\bw \rp\,dt + 
    \int_0^T b(\bze(\bu^\epsilon), \bu^\epsilon, \bw)\,dt  
    \\&= \left\langle  \bu_0, \bw(0) \right\rangle+\int_0^T\left\langle \bbf, \bw \right\rangle\,dt.
\end{align*}
Thanks to \eqref{strong_eps_conv}, the first two terms converge to their Navier-Stokes analogues.  For the nonlinear term, we estimate
\begin{align*}
&\quad
\left|\int_0^T b(\bze(\bu^\epsilon), \bu^\epsilon, \bw)\,dt
-
\int_0^T b(\bu, \bu, \bw)\,dt\right|
\\&\leq
\int_0^T |b(\bze(\bu^\epsilon)-\bu^\epsilon, \bu^\epsilon, \bw)|\,dt
+
\int_0^T |b(\bu^*, \bu^\epsilon, \bw)|\,dt
+
\int_0^T|b(\bu, \bu^*, \bw)|\,dt
\\ &\leq 
\int_0^T\Bigg(
C\epsilon^\alpha 
\normLp{2\beta}{\bu^\epsilon}^\beta 
\normLp{3}{\grad \bu^\epsilon}
\normLp{6}{\bw} \\ &\qquad\qquad+ 
\normLp{3}{\bu^*}
\normLp{2}{\grad\bu^\epsilon}
\normLp{6}{\bw}  + 
\normLp{3}{\bu}
\normLp{2}{\grad\bu^*}
\normLp{6}{\bw} \Bigg)dt
\\ &\leq 
\int_0^T\Bigg(
C_\beta \epsilon^\alpha 
\normLp{2}{\grad\bu_0}^\beta
\normLp{2}{\grad \bu^\epsilon}^\frac{1}{2}
\normLp{2}{A \bu^\epsilon}^\frac{1}{2}
\normLp{2}{\grad\bw} \\ &\qquad\qquad+ 
\normLp{2}{\bu^*}^\frac{1}{2} 
\normLp{2}{\grad\bu^*}^\frac{1}{2}
\normLp{2}{\grad\bu^\epsilon}
\normLp{2}{\grad\bw}\,  \\ &\qquad\qquad\quad+ 
\int_0^T
\normLp{2}{\bu}^\frac{1}{2}
\normLp{2}{\grad\bu}^\frac{1}{2}
\normLp{2}{\grad\bu^*}
\normLp{2}{\grad\bw} \Bigg)\, dt
\\ &\leq 
C_\beta \epsilon^\alpha 
\normLp{2}{\grad\bu_0}^{\beta}
\| A\bu^\epsilon \|_{L^2 L^2}
\| \bw \|_{L^2 V} \\ &\qquad+
C
\| \bu^* \|_{L^\infty L^2}^\frac{1}{2} 
\| \bu^* \|_{L^\infty V}^\frac{1}{2}
\normLp{2}{\grad\bu_0}
\int_0^T
\normLp{2}{\grad\bw}\, dt \\ &\qquad\quad+
\| \bu \|_{L^\infty L^2}^\frac{1}{2} 
\| \bu \|_{L^\infty V}^\frac{1}{2} 
\| \bu^* \|_{L^\infty L^2}
\int_0^T
\normLp{2}{\grad\bw} \, dt,
\end{align*}
where the three terms vanish as $\epsilon \to 0^+$ as a consequence of \eqref{lem_ineq2} and \eqref{strong_eps_conv}.
Hence, sending $\epsilon\rightarrow0^+$, and choosing $\bw\in C^\infty_c((0,T),V)$ we obtain
\begin{align}\label{NSE_fun_form}
     \dt\bu +  \nu A\bu + B(\bu, \bu) = \bbf
\end{align}
holding in the distributional sense in time with values in $V$, i.e., in the sense of $\mathcal{D}'((0,T),V')$.  But then, as in, e.g., \cite{Temam_2001_Th_Num},
Chapter 3, Lemma 1.1, 
 since the other terms in \eqref{NSE_fun_form} are in $L^2H$, it holds that $\dt\bu\in L^2H$, and moreover, \eqref{NSE_fun_form} holds in the sense of $\dt\bu\in L^2H$.  A standard argument (see, e.g., \cite{Constantin_Foias_1988,Temam_2001_Th_Num}) shows that the initial data is satisfied in the sense of $C([0,T];V)$.  That is, $\bu$ is a strong solution to the Navier-Stokes equations.

 \section{An Energy Equality for Weak Solutions}\label{sec_energy_equality}
In this section we focus only on the calmed rotational Navier-Stokes equations \eqref{cNSE_vor}. We also assume that $\bze$ satisfies condition (\ref{zeta_cond_parll}) of Definition \ref{zeta_def}, so that 
$\lp \lp \grad\times \bu \rp\times\bze(\bu) \rp \cdot \bu = 0$
in the $L^2$-sense thanks to \eqref{cross_prod_ortho}. 
\subsection{Proof of Theorem \ref{thm_energy_identity}}
Suppose $\bbf \in L^2(0,T;V')$.  Let $\bu$ be a weak solution to calmed rNSE as in Definition \ref{cNSE_weak}, with the nonlinearity given by $B(\bu,\bv) = P_\sigma\lp \lp \grad\times \bv \rp \times \bu \rp$. Taking the action of the equation in $V'$ with $\bu$ and using the Lions-Magenes Lemma\footnote{As is well-known, the Lions-Magenes Lemma is \textit{not} known to apply in the setting of weak solutions to the 3D NSE, since for those solutions, it is only known that $\dt\bu\in L^{4/3}(0,T;V')$, preventing a proof of an energy equality for weak solutions of the 3D NSE.  This seems to be an important distinction of system \eqref{cNSE_vor} from the $3$D NSE.} and the fact that $\dt\bu\in L^2(0,T;V')$, we obtain
\begin{align*}
    \frac{1}{2}\frac{d}{dt}\normLp{2}{\bu}^2 + \nu \normLp{2}{\grad\bu}^2 = \left\langle \bbf, \bu \right\rangle.
\end{align*}
 Integrating in time and using the fact that $\bu\in C([0,T];H)$, we find that, for any $t>0$,
\begin{align*}
\normLp{2}{\bu(t)}^2 + 2\nu\int_0^t \normLp{2}{\grad\bu(s)}^2\,ds = \normLp{2}{\bu_0}^2 +2\int_0^t\left\langle \bbf(s), \bu(s) \right\rangle\,ds.
\end{align*}
Therefore equations \eqref{energy_eq} and \eqref{energy_eq_integrated} are valid, proving Theorem \ref{thm_energy_identity}. \qedsymbol

\begin{remark}
Let us briefly compare system \eqref{cNSE_vor} with the 3D NSE.  For the 3D NSE, it was shown in \cite{Buckmaster_Vicol_weakNonuni_2019, Luo_Titi_weakNonuni_2020}
that weak solutions are non-unique, but it is currently a major open problem to show whether weak solutions that satisfy the energy inequality (called Leray-Hopf solutions) are unique.  In contrast, weak solutions of \eqref{cNSE_vor} are not known to be unique, but we have just shown that they satisfy not only an energy inequality, but an energy equality.  Hence, \eqref{cNSE_vor} is an example of a system which is very similar to the 3D NSE (especially given the convergence in Theorem \ref{thm_cNSE_conv}), where an energy equality is known for weak solutions but for which a proof of uniqueness of weak solutions remains elusive.
\end{remark}

\begin{remark}
It may be worth studying analogues of so-called ``suitable weak solutions,'' proposed for the 3D NSE in \cite{Duchon_Robert_2000}, for which a local energy inequality holds. This would be especially interesting for system \eqref{cNSE_vor} under assumption (\ref{zeta_cond_parll}) in Definition \ref{zeta_def} due to the point-wise vanishing of the nonlinear term.  However, we postpone this study to a future work.
\end{remark}
 \section{A Global Attractor}\label{sec_attractor}
From the existence of the energy identity \eqref{energy_eq} we are able to prove the existence of a global attractor for the dynamical system generated by solutions of calmed rNSE \eqref{cNSE_vor}. 
 \subsection{Proof of Theorem \ref{thm_cNSE_vort_Glatt}}
 Consider again the calmed rotational Navier-Stokes equations \eqref{cNSE_vor}, under conditions \ref{zeta_cond_Lip}, \ref{zeta_cond_bdd}, \ref{zeta_cond_conv}, and \ref{zeta_cond_parll} of Definition \ref{zeta_def}.
Take $\bbf \in H$ to be time-independent, and for a given $R > 0$,  
let $B_R = \{ \bu \in H: \normLp{2}{\bu} \leq R \}$. Now choose $ \bu_0 \in B_R$. 
On the right hand side of \eqref{energy_eq}, we use H\"older's, Poincar\'e's, and Young's inequalities to obtain 
\begin{align}\notag
    \abs{\lp \bbf, \bu \rp} &\leq 
    \frac{1}{2\nu\lambda_1}\normLp{2}{\bbf}^2 + \frac{\nu}{2}\normLp{2}{\grad\bu}^2
\end{align}

We insert the second estimate into the first and rearrange the terms, which yield
\begin{align}\label{cNSE_Glatt1}
\frac{d}{dt}\normLp{2}{\bu}^2 + \nu \normLp{2}{\grad\bu}^2 \leq \frac{1}{\nu\lambda_1}\normLp{2}{\bbf}^2.
\end{align}
We apply Poincar\'e's inequality once more, 
\begin{align} \notag
    \frac{d}{dt}\normLp{2}{\bu}^2 + \nu\lambda_1\normLp{2}{\bu}^2 \leq \frac{1}{\nu\lambda_1}\normLp{2}{\bbf}^2,
\end{align} 
then we apply Gr\"onwall's inequality:
\begin{align} \notag
    \normLp{2}{\bu(t)}^2 & \leq 
    e^{-\nu\lambda_1 t}\normLp{2}{\bu_0}^2 + 
    \frac{1}{\nu\lambda_1}\lp 1 - e^{-\nu\lambda_1 t} \rp\normLp{2}{\bbf}^2 \\ &\leq \notag
    e^{-\nu\lambda_1 t}R^2 + 
    \frac{1}{\nu\lambda_1}\lp 1 - e^{-\nu\lambda_1 t} \rp\normLp{2}{\bbf}^2
\end{align}
We now set 
\[t_0 = \frac{1}{\nu\lambda_1}\ln( 1 + R^2 ), \]
so that 
\[
\max\left\{ e^{-\nu\lambda_1 t}, e^{-\nu\lambda_1 t}R^2 \right\} < 1
\]
for all $ t \geq t_0$. Then we obtain
\begin{align} \label{cNSE_Glatt2}
    \normLp{2}{\bu(t)}^2 < \rho_0
\end{align}
for all $t \geq t_0$, where 
$\rho_0 = 1 + \frac{1}{\nu\lambda_1} \normLp{2}{\bbf}^2$.

If instead we integrate \eqref{cNSE_Glatt1} on the interval\footnote{Here, the ``$1$'' in ``$t-1$'' has dimensions of time.  Instead, one could consider the interval $[t-\tau, t]$, where $\tau=\frac{1}{\nu\lambda_1}$, but we use a unit interval to simplify the presentation.} $[t-1, t]$ for some $t \geq t_0 + 1$, we obtain 
\begin{align}\notag
    \normLp{2}{\bu(t)}^2 + 
    \nu \int_{t-1}^t \normLp{2}{\grad\bu}^2 \leq 
    \normLp{2}{\bu(t-1)}^2 + 
    \frac{1}{\nu\lambda_1}\normLp{2}{\bbf}^2,
\end{align}
from which we deduce, by \eqref{cNSE_Glatt2},
\begin{align}\label{cNSE_Glatt3}
    \int_{t-1}^t \normLp{2}{\grad\bu}^2 ds \leq \rho_1,
\end{align}
where $\rho_1 = \frac{1}{\nu}\rho_0 + \frac{1}{\nu^2\lambda_1}\normLp{2}{\bbf}^2$.
Now, we take the action of \eqref{cNSE_vor_mo} with $-\lap \bu$, and use the Lions-Magenes Lemma to obtain 
\begin{align} \notag
    & \quad
    \frac{1}{2}\frac{d}{dt}\normLp{2}{\grad \bu}^2 + 
    \nu\normLp{2}{\lap\bu}^2 \\ &= \notag
    \lp \lp \grad\times \bu \rp \times \bze(\bu), \lap \bu \rp - 
    \lp \bbf, \lap\bu \rp
    \\ & \leq \notag
    C_{\nu}\normLp{\infty}{\bze} \normLp{2}{\grad\bu}^2 + 
    C_\nu \normLp{2}{\bbf}^2 + 
    \frac{\nu}{2}\normLp{2}{\lap \bu}^2
\end{align}
We then rearrange the inequality above which yields 
\begin{align} \label{cNSE_Glatt4}
    \frac{d}{dt}\normLp{2}{\grad\bu}^2 + 
    \nu \normLp{2}{\lap\bu}^2 \leq 
    C_{\nu}\normLp{\infty}{\bze}  \normLp{2}{\grad\bu}^2 + 
    C_\nu \normLp{2}{\bbf}^2.
\end{align}
Now, select $s$ and $t$ such that $t > t_0 +1 $ and $t-1 < s < t$. We remove the viscous term from the left-hand side, then integrate \eqref{cNSE_Glatt4} on the interval $[s,t]$ and apply \eqref{cNSE_Glatt3} to obtain
\begin{align} \label{cNSE_Glatt5}
    \normLp{2}{\grad\bu(t)}^2 
    \leq 
    \normLp{2}{\grad\bu(s)}^2 +
    C_\nu \normLp{2}{\bbf}^2 + 
    C_{\nu}\normLp{\infty}{\bze} \rho_1.
\end{align}
Integrating once more in $s$ on the interval $[t-1, t]$ and again using \eqref{cNSE_Glatt3}, it follows that, for $ t > t_0 + 1$,
\begin{align}\label{cNSE_Glatt6}
    \normLp{2}{\grad\bu(t)}^2 \leq 
    \rho_2,
\end{align}
where $ \rho_2 = \rho_1 + C_\nu \normLp{2}{\bbf}^2 + C_{\nu}\normLp{\infty}{\bze} \rho_1$. From this inequality we deduce that $B_{\rho_2} = \left\{ \bu \in H: \normLp{2}{\bu} \leq \rho_2\right\}$ is bounded in $V$. Since $V$ is compactly embedded in $H$, we deduce that $B_{\rho_2}$ is a compact absorbing set in $H$. 
Applying Theorem $10.5$ of \cite{Robinson_2001}, we conclude that there exists a global attractor in $H$. \qedsymbol

\begin{remark}
    Observe that the upper bounds in \eqref{cNSE_Glatt4}, \eqref{cNSE_Glatt5}, \eqref{cNSE_Glatt6} each depend on $\normLp{\infty}{\bze}$. Therefore these upper bounds do not remain valid as $\epsilon \to 0^+$, since $\displaystyle\lim_{\epsilon \to 0^+ }\normLp{\infty}{\bze} = \infty$.
\end{remark}

 \section{Conclusions}

 We proposed two modifications of the $3$D Navier-Stokes equations: one involved a modification to the advective velocity term of Navier-Stokes (with kinematic pressure), which we refer to as `calmed Navier-Stokes,' and the other involves a modification to the Lamb vector of Navier-Stokes (with Bernoulli pressure), which we term `calmed rotational Navier-Stokes.' We have successfully demonstrated the existence of weak solutions for both of these calmed systems, although the question of whether these solutions are unique remains open. Furthermore, we have established the global well-posedness for strong solutions in both cases. Moreover, we demonstrate that calmed strong solutions do converge to strong solutions of the Navier-Stokes equations on sufficiently small time intervals, provided suitable conditions on the calming function and suitable regularity of the solution to Navier-Stokes.  

In the context of the calmed rotational Navier-Stokes Equations (for suitable calming functions), we also establish the existence of an energy identity and the presence of a compact global attractor within the function space $H$.

\section*{Appendix}
\subsection*{Proof of Proposition \ref{zeta_cond_prop}}
For $\bze_1$ and $\bze_2$, the proof follows from direct computation. For $\bze_3$, we first note that $\bze_3(\bx)$ is increasing towards $\lp \frac{\pi}{2\epsilon}, \frac{\pi}{2\epsilon}, \frac{\pi}{2\epsilon} \rp^T$ as $x_1, x_2, x_3$ get large. It follows that $\normLp{\infty}{ \bze_3} = \frac{\sqrt{3}\pi}{2\epsilon}$. For the estimate on pointwise convergence,
    we begin by noting that, for $\bx = \lp x_1, x_2, x_3 \rp^T$,
    \[
    \abs{\frac{d}{dx_i}\lp \frac{1}{\epsilon}\arctan(\epsilon x_i) - x_i \rp} \leq 
    \epsilon^2 x_i^2
    \]
    for $i = 1,2,3$.
    Thus 
    \[
    \abs{\frac{1}{\epsilon} \arctan(\epsilon x_i) - x_i} \leq 
    \frac{1}{3} \epsilon^2 x_i^3, 
    \]
    and since $\abs{\lp x_1^3, x_2^3, x_3^3 \rp^T} \leq 3 \abs{\bx}^3$, we have
    \[
    \abs{\bze_3(\bx) - \bx} \leq \epsilon^2 \bx^3.
    \]
    For $\bze_4$, determining the upper bound is trivial. Following a similar procedure as above for obtaining our pointwise convergence estimate, one checks that for all $r\geq 0$,
    \[
    \abs{\frac{d}{dr}\lp q^\epsilon(r) - r\rp} \leq 2\epsilon r
    \]
    for $q^\epsilon(r)$ defined in $\eqref{calm4}$.
    It follows that 
    $\abs{q^\epsilon(r) - r } \leq \epsilon r^2$, hence 
    \[\abs{\bze_4(\bx) - \bx} \leq \epsilon\abs{\bx}^2.\]
    \qedsymbol
 \section*{Acknowledgments}
 \noindent
The research of A.L. and M.E. was supported in part by NSF Grants CMMI-1953346 and DMS-2206741.  The research of A.L. was also supported in part USGS  grant G23AS00157 number GRANT13798170. J.W. was partially supported by NSF Grants DMS 2104682 and DMS 2309748.
 
\begin{scriptsize}

\end{scriptsize}

\end{document}